\documentclass[a4paper,11pt]{amsart}

\usepackage{MyPack}
\usepackage{MyMathPack}
\newcommand{\md}[1]{\emph{#1}}
	
\begin{document}

\title{Decomposition of Topological Azumaya Algebras with Orthogonal Involution}

\author[Arcila-Maya]{Niny Arcila-Maya}

\date{\today}

\address{Department of Mathematics, San Francisco State University, San Francisco~CA, 94132 USA}

\email{niny.arcilamaya@sfsu.edu}

\begin{abstract}
Let $\cA$ be a topological Azumaya algebra of degree $mn$ with an orthogonal involution over a CW complex $X$ of dimension less than or equal to $\min\{m,n\}$. We give conditions for the positive integers $m$ and $n$ so that $\cA$ can be decomposed as the tensor product of topological Azumaya algebras of degrees $m$ and $n$ with orthogonal involutions.
\end{abstract}

\maketitle

\section{Introduction}

The concept  of a central simple algebra over a field was generalized by Azumaya \cite{Azu1951} and Auslander--Goldman \cite{AG1960} to the notion of an Azumaya algebra over a commutative ring, and Grothendieck \cite{GroI1966} defined this concept in the context of a locally-ringed topos, that is, in a category of sheaves equipped with a designated choice of sheaf of rings with enough points, $\cO$, having local rings as stalks. One recovers the case of Auslander--Goldman when the category is $(\Spec R)_{\et}$ of étale sheaves over the spectrum of the ring $R$. That is, an Azumaya algebra of degree $n$ is a sheaf of $\cO$ algebras that is locally isomorphic to $\M(n,\cO)$. In the topological case, a \textit{topological Azumaya algebra of degree $n$} is defined by taking the sheaf $\cO$ to be the sheaf of continuous functions with value $\CC$, that is, it is a bundle of associative and unital complex algebras over a topological space that is locally isomorphic to the matrix algebra $\fM{n}$, \cite[1.1]{GroI1966}.

\begin{example}
The endomorphism bundle of a complex vector bundle of rank $n$ is a topological Azumaya algebra of degree $n$ over $X$.
\end{example}
\begin{example}

Let $G$ be a group acting on the left on $\fM{n}$ via algebra automorphisms. Let $X$ be a topological space with a free right $G-$action. Suppose the projection $p:X\to X/G$ is a principal $G-$bundle. Then the fiber bundle with fiber $\fM{n}$ associated to $p$, $\fM{n} \to X\times^{G}\fM{n} \to X/G$ is a topological Azumaya algebra of degree $n$ over $X/G$, where $G$ acts on $X\times\fM{n}$ by $(x,M)\cdot g\coloneqq \bigl(x\cdot g,g\cdot M\bigr)$ for all $x\in X$ and $M\in \fM{n}$.
\end{example}

In the theory of central simple algebras over a field $k$, a theorem of Wedderburn states that any central simple algebra $A/k$ has the form $\M(n,D)$, where $D/k$ is a division algebra that is unique up to isomorphism, \cite[Theorem 1.3]{Sdiv1999}. This theorem reduces the classification problem for finite-dimensional central simple algebras over $k$ to the classification problem for finite-dimensional central division algebras over $k$. Two finite-dimensional central simple algebras $A\iso\M(m,D)$ and $B\iso\M(n,E)$ are Brauer equivalent if their division algebras $D$ and $E$ are isomorphic. The Brauer group of $k$, $\Br(k)$, is the group of equivalence classes of finite-dimensional central simple algebras over $k$ modulo Brauer equivalence, with the tensor product of algebras, \cite[Proposition 4.3]{FD1993}. 

Given two topological Azumaya algebras $\cA$ and $\cA'$ over $X$ of degrees $m$ and $n$, respectively, we can define the tensor product $\cA\tensor \cA'$ by performing the tensor product of (Kronecker product of) matrices $\fM{m}\tensor_{\CC}\fM{n}$ fiberwise. We say two topological Azumaya algebras $\cA$ and $\cA'$ are \md{Brauer equivalent} if there exist complex vector bundles $\cV$ and $\cV'$ such that $\cA \tensor \End(\cV) \iso \cA' \tensor \End(\cV')$ as bundles of $\CC$-algebras. The \md{(topological) Brauer group of $X$}, $\Br(X)$, is the set of isomorphism classes of topological Azumaya algebras over $X$ modulo Brauer equivalence with the tensor product operation.  The identity of $\Br(X)$ is the class $[\End(\cV)]$ for all complex vector bundles $\cV$, and $[\cA]^{-1}=[\cA^{\text{op}}]$ for all classes $[\cA] \in \Br(X)$ since $\cA\tensor \cA^{\text{op}}\iso \End(\cA)$. We say a topological Azumaya $\cA$ is \md{Brauer trivial} if its class $[\cA] \in \Br(X)$ is equal to the identity.

If $D/k$ is a finite-dimensional central division algebra, then $\dim_{k}(D)$ is a square, \cite[Theorem 3.10]{FD1993}. The \md{degree} of $D/k$ is defined by $\displaystyle{\sqrt{\dim_{k}(D)}}$. The theory of central division algebras is equipped with a structure theorem stating that every finite-dimensional central division algebra $D/k$ can be broken up into pieces corresponding to the prime factorization of its degree $\deg(D)=p_{1}^{n_{1}}\cdots p_{r}^{n_{r}}$, i.e. $D$ is isomorphic to $D_{1}\tensor_{k} D_{2} \tensor_{k} \cdots \tensor_{k} D_{r}$ where each $D_{i}$ is a division algebra of degree $p_{i}^{n_{i}}$, and the decomposition is unique up to isomorphism, \cite[Theorem 5.7]{Sdiv1999} and \cite[Proposition 4.5.16]{GS06}. Saltman asked in \cite[page 35]{Sdiv1999} whether the analogue to the prime decomposition theorem for central simple algebras over a field holds for Azumaya algebras over a commutative ring. Using algebraic topology and topological Azumaya algebras, Antieau--Williams showed that, in general, there is no prime decomposition for these algebras, \cite[Corollary 1.3]{AW2x32014}. The author provided conditions for a positive integer $n$ and a topological space $X$ such that a topological Azumaya algebra of degree $n$ on $X$ has a non-unique tensor product decomposition, \cite[Theorem1.3, Remark 3.7]{TAA2021}.

An \md{involution} on a finite-dimensional central simple algebra $A/k$ is an anti-automorphism $\tau:A\to A$. In other words $\tau\circ \tau=\id_{A}$, $\tau(a+b)=\tau(a)+\tau(b)$, and  $\tau(ab)=\tau(b)\tau(a)$ for all for $a, b \in A$. A central simple algebra with an involution is denoted by $(A,\tau)$. It can be checked that the center $k$ is preserved under $\tau$. The restriction of $\tau$ to $k$ is therefore an automorphism which is either the identity or of order $2$. Involutions that leave the center invariant are called involutions of the \md{first kind}. Involutions whose restriction to $k$ is an automorphism of order $2$ are called involutions of the \md{second kind}, \cite{KMRT1998}. A \md{quaternion algebra} over a field $k$ is a central simple $k-$algebra of degree $2$. Any quaternion algebra has a symplectic involution called the \md{quaternion conjugation}, \cite[page 26]{KMRT1998}.  Let $(A_{1},\tau_{1}), \dots, (A_{n},\tau_{n})$ be central simple $k-$algebras with involution of the first kind. Then $\tau_{1}\tensor\cdots \tensor\tau_{n}$ is an involution of the first kind on $A_{1}\tensor_{k}\cdots \tensor_{k} A_{n}$, \cite[Proposition 2.23(1)]{KMRT1998}. Knus--Parimala--Srinivas gave a necessary and sufficient decomposability condition for an involution on a tensor product of two quaternion algebras, \cite{KPS1991}. Merkurjev showed that every central simple algebra with involution is Brauer equivalent to a tensor product of quaternion algebras, \cite{Merku}. However, Amitsur--Rowen--Tignol provided examples of division algebras of degree $8$ with involutions which do not decompose into tensor products of quaternion algebras, \cite{ART1979}.

Knus--Parimala--Srinivas generalized the notion of a central simple algebras with an involution to Azumaya algebras over schemes, \cite{KPS1990}. Saltman presented a classification of involutions of Azumaya algebras over commutative rings into kinds, \cite[Section 3]{SAzuInvo1978}. An involution on a topological Azumaya algebra is said to be an \textit{involution of the first kind} or an \textit{involution of the second kind} depending on whether or not there is a group action on the base space, respectively. All the involutions we discuss here are involutions of the first kind, which are classified as orthogonal and symplectic involutions.

\begin{definition}
Let $X$ be a connected topological space, and let $\cA$ be a topological Azumaya algebra of degree $n$ over $X$. An \textit{involution on $\cA$}  is a morphism of fiber bundles $\tau:\cA \rightarrow \cA$ such that $\tau\circ\tau=\id_{\cA}$, and when restricted to fibers it is an anti-automorphism of complex algebras. In this case, $(\cA,\tau)$ is called a \textit{topological Azumaya algebra with involution}.
\end{definition}

\begin{definition}
Let $X$ be a connected topological space, and let $(\cA,\tau)$ be a topological Azumaya algebra with involution over $X$. The involution $\tau$ is said to be \md{orthogonal} (\md{symplectic}) if the restriction $\tau|_{\cA^{-1}(x)}:\cA^{-1}(x)\to\cA^{-1}(x)$ is an orthogonal (a symplectic) involution of complex algebras for all $x\in X$, in the sense of the definition given in Subsection \ref{invoma}.
\end{definition}

In this paper, we study decomposition properties of topological Azumaya algebras endowed with orthogonal involutions, generalizing prior results from \cite{TAA2021}, which considered the case without involutions. The decomposition arguments require the degrees $m$ and $n$ to be relatively prime, which in particular rules out the even-even case and leaves two parity cases: the mixed-parity case (say $m$ even and $n$ odd, or vice versa) and the odd-odd case. Theorem \ref{mainO} treats the mixed-parity providing conditions $m$, $n$ and a topological space $X$ under which a topological Azumaya algebra of degree $mn$ over $X$ with an orthogonal involution decomposes into a tensor product of topological Azumaya algebras each carrying an orthogonal involution:
\begin{theorem}\label{mainO}
Let $m$ and $n$ be relatively prime positive integers such that $m$ is even, and $n$ is odd. Let $X$ be a CW complex such that $\dim(X)\leq d$ where $d\coloneqq\min\{m,n\}$. If $\cA$ is a topological Azumaya algebra of degree $mn$ over $X$ with an orthogonal involution, then there exist topological Azumaya algebras $\cA_{m}$ and $\cA_{n}$ of degrees $m$ and $n$, respectively, such that $\cA_{m}$ and $\cA_{n}$ have orthogonal involutions, $\cA_{n}$ is Brauer-trivial and $\cA\iso \cA_{m}\tensor\cA_{n}$. 
\end{theorem}

In the odd-odd case, the appropriate statement is no longer formulated in terms of topological Azumaya algebras. Instead, it becomes a decomposition result for orthogonal complex vector bundles, proved in Theorem \ref{mainObundles}:
\begin{theorem}\label{mainObundles}
Let $m$ and $n$ be relatively prime positive integers such that $m$ and $n$ are odd. Let $X$ be a CW complex such that $\dim(X)\leq \min\{m,n\}$. If $\cV$ is an orthogonal complex vector bundle of rank $mn$ over $X$, then there exist orthogonal complex vector bundles $\cV_{m}$ and $\cV_{n}$ of ranks $m$ and $n$, respectively, such that $\cV\iso \cV_{m}\tensor\cV_{n}$. 
\end{theorem}

It is worth observing that Theorem \ref{mainO} does not hold for CSAs with orthogonal involution. Specifically, for any $m, n > 1$, there exists a Brauer-trivial CSA with an orthogonal involution that does not decompose as the tensor product of degree-$m$ and degree-$n$ CSAs with orthogonal involution. Assuming the base field's characteristic is not $2$ for simplicity, if all Brauer-trivial CSAs with orthogonal involution of degree $mn$ could be decomposed as stated, then any $mn$-dimensional quadratic form would decompose as the tensor product of an $m$-dimensional and an $n$-dimensional quadratic form. However, the essential dimension of a ``generic'' $mn$-dimensional quadratic form is $mn$, whereas the essential dimension of the tensor product of an $m$-dimensional and an $n$-dimensional quadratic form is at most $m+n-1$. Hence, this is impossible whenever $mn>m+n-1$, which is the case whenever $m, n > 1$.

From \cite[8.2]{Steen1951} there is a bijective correspondence
\begin{align*}\label{bcTAAwithInvo}
\begin{Bmatrix} 
\text{Isomorphism classes of degree-$n$}\\
\text{topological Azumaya algebras}\\
\text{over $X$ with an involution}\\
\text{locally isomorphic to $\tau$}
\end{Bmatrix}
\leftrightarrow
\begin{Bmatrix} 
\text{Isomorphism classes of principal}\\
\text{$\Aut\bigl(\fM{n},\tau\bigr)$-bundles}\\
\text{over $X$}
\end{Bmatrix},
\end{align*}
where $\tau$ is an involution of the first kind on $\fM{n}$. Proposition \ref{P:OSp}
implies that topological Azumaya algebras over $X$ with involution are classified by $[X,\B\fPO{n}]$ in the orthogonal case, and $[X,\B\fPSp{n}]$ in the symplectic case, where $\B\fPO{n}$ and $\B\fPSp{n}$ are the classifying spaces of the projective complex orthogonal group of degree $n$ and the projective complex sympletic group of degree $2n$, respectively, see Subsection \ref{notation}.

We prove in Theorem \ref{mainO} that a map $X\rightarrow \B\fPO{mn}$ can be lifted to $\B\fPO{m}\times \B\fSO{n}$ via the map $f_{\tensor}\colon \B\fPO{m}\times \B\fSO{n}\to \B\fPO{mn}$, when the dimension of $X$ is less than $d+1$. The proof of Theorem \ref{mainO} relies in the description of the homomorphisms induced on homotopy groups by the $r$-fold direct sum of matrices $\oplus^{r}:\fO{n}\to\fO{rn}$ in the range $\{0,1,\dots,n-1\}$. We call this set \textit{the stable range for $\fO{n}$}.

Section \ref{sec:prelim} introduces the tensor product operations for compact Lie groups related to the complex symplectic and complex orthogonal groups. In Section \ref{sec:O} we review the homotopy groups of the complex orthogonal and complex projective orthogonal groups. We introduce the stabilization maps, and establish its connectivity. We also present results describing
the homomorphisms induced on homotopy groups by the operations defined in Section \ref{sec:prelim}. Section \ref{sec:mainO} is devoted to the proof of Theorem \ref{mainO}, and 
to provide examples of topological Azumaya algebras with orthogonal involutions that do not decompose as the tensor product of topological Azumaya algebras with orthogonal involutions when defined over spaces whose dimension is outside the stable range for orthogonal groups. Finally, in Section \ref{mainObundles} we prove the decomposition of complex orthogonal vector
bundles in the stable range for the complex orthogonal group.

\subsection{Acknowledgments}
The author would like to express her thankfulness to \textit{Ben Williams} for proposing this research topic, pointing out relevant references, and having devoted a great deal of time to discuss details of the research with the her. She also gratefully acknowledges \textit{Kirsten Wickelgren} and \textit{Diarmuid Crowley} for reading the introduction and providing valuable suggestions that improve its exposition. I am thankful to the anonymous referee for their thorough review of the paper and for offering insightful comments. Many of their suggestions have been incorporated into the revised version of the manuscript.

\subsection{Notation} \label{notation}

Throughout this paper, all topological spaces will have the homotopy type of a CW complex. We fix basepoints for connected topological spaces, and for topological groups we take the identities as basepoints. We write $\pi_{i}(X)$ in place of $\pi_{i}(X,x_{0})$.

We adopt the notation for the classical Lie groups as in \cite{MiToToG1991}. We denote by $\fM{n}$ the set of $n\times n$ complex matrices. We use $\fGL{n}$, $\fO{n}$, $\fSO{n}$, and $\fSp{n}$ to denote the complex general linear group of degree $n$, the complex orthogonal group of degree $n$, the special complex orthogonal group of degree $n$, and the complex symplectic group of degree $2n$, respectively. These groups are closed subgroups of $\fGL{n}$ in the orthogonal cases, and of $\fGL{2n}$ in the symplectic case. In matrix terms these groups are defined as
\begin{align*}
&\fGL{n}=\bigl\{M \in \fM{n}:\; M \text{ is invertible}\bigr\},\\
&\fO{n}=\bigl\{M \in \fM{n}: \; M^{\tr}M=MM^{\tr}=I_{n}\bigr\},\\
&\fSO{n}=\bigl\{M \in \fO{n}: \; \det(M)=1\bigr\},\; \text{ and } \;\\
&\fSp{n}=\bigl\{M \in \fM{2n}: \; M^{\tr}J_{2n} M=J_{2n} \bigr\},
\end{align*}
where $\tr$ denotes transposition, and
\[
J_{2n}=%
\begin{pmatrix}
0 & I_{n}\\
-I_{n} & 0
\end{pmatrix}.
\]
If $G$ is a group, then $Z(G)$ denotes its center. The projective complex orthogonal group of degree $n$, the projective special complex orthogonal group of degree $n$, and the projective complex sympletic group of degree $2n$ shall be denoted by $\fPO{n}\coloneqq \fO{n}/Z(\fO{n})$, $\fPSO{n}\coloneqq \fSO{n}/Z(\fSO{n})$, and $\fPSp{n}\coloneqq \fSp{n}/Z(\fSp{n})$, respectively. 

Given matrices $A_{1}, \dots, A_{r} \in \M(n,\CC)$, we write $\diag(A_{1},\dots,A_{r}) \in \M(rn,\CC)$ to denote the block-diagonal matrix with blocks $A_{1}, \dots, A_{r}$ along the diagonal.

\section{Preliminaries}  \label{sec:prelim}

\subsection{Involutions on matrix algebras} \label{invoma}

Let $A \in \fGL{n}$, we denote by $\Inn_{A}:\fM{n}\to\fM{n}$ the inner automorphism induced by $A$, $\Inn_{A}(M)=AMA^{-1}$ for all $M\in \fM{n}$.

Observe that transposition, $\tr:\fM{n} \to \fM{n}$, is an involution of the first kind on $\fM{n}$. All automorphisms of $\bigl(\fM{n},\tr\bigr)$ arise by conjugation by an invertible matrix: Let $\tau$ be an arbitrary involution on $\fM{n}$. Since $\tr\circ \,\tau$ is an automorphism of $\bigl(\fM{n},\tr\bigr)$, then $\tr\circ\, \tau=\Inn_{A}$ for some $A \in \fGL{n}$. Hence $\tau(M)=A^{-\tr}M^{\tr}A^{\tr}$. Since $\tau$ is an involution, then $M=(A^{-\tr}A)M(A^{-1}A^{\tr})$. Therefore $\Inn_{A^{-1}A^{\tr}}=\id_{\fM{n}}$, this is $A^{-1}A^{\tr}=\lambda I_{n}$ for some $\lambda \in \CC$. This last equation implies $A^{\tr}=\lambda A$, and thus $A=\lambda^{2}A$. Hence $\lambda=1$ or $\lambda=-1$, i.e. $A^{\tr}=A$ or $A^{\tr}=-A$. 

Consider the subspaces
\begin{align*}
\Bigl(\fM{n},\tau\Bigr)^{+}&=\Bigl\{M\in \fM{n}:\;\;\tau(M)=M\Bigr\} \quad \text{(symmetric elements), and}\\
\Bigl(\fM{n},\tau\Bigr)^{-}&=\Bigl\{M\in \fM{n}:\;\;\tau(M)=-M\Bigr\} \quad \text{(skew symmetric elements).}
\end{align*}

Observe that 
\[
\Bigl(\fM{n},\tau\Bigr)^{+}=%
\begin{cases}
\Bigl(\fM{n},\tr\Bigr)^{+}A & \text{if $A^{\tr}=A$},\\
\Bigl(\fM{n},\tr\Bigr)^{-}A & \text{if $A^{\tr}=-A$}.
\end{cases}
\]
where $\dim_{\CC}\Bigl(\fM{n},\tr\Bigr)^{+}=\frac{1}{2}n(n+1)$ and $\dim_{\CC}\Bigl(\fM{n},\tr\Bigr)^{-}=\frac{1}{2}n(n-1)$.

In summary, we obtain the following proposition.
\begin{proposition}\label{P:OSp}
Let $\tau$ be an involution of the first kind on $\fM{n}$.
\begin{enumerate}
\item The involution $\tau$ on $\fM{n}$ has the form $\tau=\Inn_{A}\circ \tr$ for some $A\in \fGL{n}$ such that $A^{\tr}=\pm A$.

\item The subspace $\Bigl(\fM{n},\tr\Bigr)^{+}$ has complex dimension $\frac{1}{2}n(n+1)$ if and only if $A^{\tr}=A$.

\item The subspace $\Bigl(\fM{n},\tr\Bigr)^{+}$ has complex dimension $\frac{1}{2}n(n-1)$ if and only if $A^{\tr}=-A$.
\end{enumerate}
\end{proposition}

If $\tau$ is an involution on $\fM{n}$, then the \textit{type} of $\tau$ is said to be \textit{orthogonal} if $A^{\tr}=A$, and \textit{symplectic} if $A^{\tr}=-A$. Up to isomorphism, the matrix algebra $\fM{n}$ can carry at most one involution (transposition) if $n$ is odd; and up to two involutions (orthogonal and symplectic) if $n$ is even, \cite[Proposition I.2.20]{KMRT1998}.

\subsection{Non-degenerate bilinear forms} \label{tpoBf}
Let a pair $(V,B)$ denote a finite-dimensional complex vector space, $V$, with a non-degenerate bilinear form, $B$. If $\dim_{\CC}V$ is even, $V$ can be given both a skew-symmetric bilinear form and a symmetric bilinear form.  If $\dim_{\CC}V$ is odd, $V$ can be given a symmetric bilinear form. These bilinear forms are unique up to isomorphism, given that $\CC$ is algebraically closed, \cite[Theorem 9.13]{Car2017}.

The automorphism group of $(V,B)$ is
\[
\Aut(V,B)=\Bigl\{T \in \fGL{n}: \; B(Tv,Tw)=B(v,w) \text{ for all } v,w \in V\Bigr\}.
\]

Let $(\CC^{n}, B)$ and $(\CC^{2n}, B')$ denote the spaces $\CC^{n}$ and $\CC^{2n}$ with the standard symmetric bilinear form and the standard skew-symmetric bilinear form, respectively. Then 
\[
\Aut(\CC^{n}, B) \iso \fO{n} \quad \text{ and } \quad \Aut(\CC^{2n}, B') \iso \fSp{n}.
\]

Let $(V,B)$ and $(V,B')$ be finite-dimensional complex vector spaces with non-degenerate bilinear forms. The tensor product $(V, B)\tensor (V',B')$ is a non-degenerate bilinear form over $\CC$, where $(B \tensor B')(v\tensor v',w\tensor w')=B(v,w)B'(v',w')$ (multiplication in $\CC$), \cite[Section 1.21, Section 1.27]{Gre1967}. The tensor product is bifunctorial, so that one obtains induced maps
\begin{equation}\label{tensorAut}
\begin{tikzcd}
\tensor: \Aut(V,B) \times \Aut(V',B') \arrow[r] &\Aut(V \tensor V', B \tensor B').
\end{tikzcd}
\end{equation}

The following tensor product operation is of interest in the subsequent sections. Let $(\CC^{m}, B)$ and $(\CC^{n}, B')$ be the spaces $\CC^{m}$ and $\CC^{n}$ with the standard symmetric bilinear forms. Then, the tensor product \eqref{tensorAut} induces a homomorphism 
\begin{equation*}
\begin{tikzcd}
\tensor:\fO{m} \times \fO{n} \arrow[r] & \fO{mn}.
\end{tikzcd}
\end{equation*}

\begin{remark}
The tensor product \eqref{tensorAut} also induces the following tensor products on Lie groups.
\begin{enumerate}
\item \label{SpO} Let $(\CC^{2m}, B)$ be the space $\CC^{2m}$ with the standard skew-symmetric bilinear form, and $(\CC^{n}, B')$ be the space $\CC^{n}$ with the standard symmetric bilinear form. Then, the tensor product \eqref{tensorAut} induces a homomorphism 
\begin{equation*}
\begin{tikzcd}
\tensor:\fSp{m} \times \fO{n}\arrow[r] & \fSp{mn}.
\end{tikzcd}
\end{equation*}

\item \label{SpSp} Let $(\CC^{2m}, B)$ and $(\CC^{2n}, B')$ be the spaces $\CC^{2m}$ and $\CC^{2n}$ with the standard skew-symmetric bilinear forms. From the tensor product \eqref{tensorAut} there is a homomorphism 
\begin{equation*}
\begin{tikzcd}
\tensor:\fSp{m} \times \fSp{n} \arrow[r] & G,
\end{tikzcd}
\end{equation*}
where $G\leq \fGL{4mn}$ denotes $\Aut(\CC^{4mn}, B \tensor B')$. Observe that $B\tensor B'$ is not the standard symmetric bilinear form on $\CC^{4mn}$. Let $P \in \fGL{4mn}$  be matrix associated to changing the basis of $\CC^{4mn}$ to an orthonormal basis. Then we have the commutative square below
\begin{equation*}
\begin{tikzcd}[row sep=large]
G \arrow[r,hookrightarrow] \arrow[d,"\Inn_{P}"] & \GL(4mn,\CC) \arrow[d,"\Inn_{P}"]\\
\Or(4mn,\CC) \arrow[r,hookrightarrow] \arrow[u,leftarrow,"\iso"] & \GL(4mn,\CC). \arrow[u,leftarrow,"\iso"]
\end{tikzcd}
\end{equation*}

Thus, the composite $\Inn_{P}\circ \,\tensor:\fSp{m} \times \fSp{n} \to G \to \fO{4mn}$ gives a homomorphism 
\begin{equation*}
\begin{tikzcd}
\stensor:\fSp{m} \times \fSp{n}\arrow[r] &  \fO{4mn}.
\end{tikzcd}
\end{equation*}
\end{enumerate}
\end{remark}

To simplify notation, in the following sections we write $\Or(n)$, $\SO(n)$, $\PO(n)$, and $\PSO(n)$ in place of $\fO{n}$, $\fSO{n}$, $\fPO{n}$, and $\fPSO{n}$, respectively, except in subsection titles.

\section{Stabilization of operations on the complex orthogonal group} \label{sec:O}

In this section, we recall the homotopy groups of $\Or(n)$ and $\SO(n)$ in the stable range, and their first unstable homotopy groups. We compute the homotopy groups of $\PO(n)$ and $\PSO(n)$ in low degrees.

For $n=1,2$ we have:
\begin{align*}
&\Or(1)=S^{0}, & &\PO(2)\iso \Or(2),\\
&\SO(1)=\{1\},  & & \SO(2)\iso \U(1)\iso S^{1}, \quad \text{and}\\
&\PO(1)=\PSO(1)=\{1\}, & & \PSO(2)\iso S^{1}.
\end{align*}

The group $\SO(n)$ is the connected component of the identity of $\Or(n)$, hence $\pi_{0}(\SO(n))$ is trivial for all $n\geq 1$. Moreover, for $n\geq 2$, the determinant map fits into a fiber sequence $\SO(n)\hookrightarrow \Or(n) \xrightarrow{\det} \ZZ/2$, where $\ZZ/2$ is given the discrete topology. The associated long exact sequence in homotopy then shows that
\[
\pi_{i}(\SO(n))\iso\pi_{i}(\Or(n)) \quad \text{ for all } \quad i\geq 1.
\]

For $n\geq 3$ and $i<n-1$, Bott periodicity yields
\[
\pi_{i}(\Or(n))\iso
\begin{cases}
0 & \text{if $i\equiv2,4,5,6$ (mod 8),}\\
\ZZ/2 & \text{if $i\equiv0,1$ (mod 8),}\\
\ZZ & \text{if $i\equiv3,7$ (mod 8).}
\end{cases}
\]

To compute the homotopy groups of $\PO(n)$ and $\PSO(n)$ for $n\geq 3$, consider the diagrams \eqref{cd:Oeven} and \eqref{cd:Oodd} below for $k\geq 2$.
\begin{equation}\label{cd:Oeven}
\begin{tikzcd}[column sep=large]
\{\pm I_{2k}\} \arrow[equal,r] \arrow[hookrightarrow,d] & \{\pm I_{2k}\}  \arrow[r] \arrow[hookrightarrow,d] & \{1\} \arrow[d] \\
\SO(2k) \arrow[hookrightarrow,r] \arrow[twoheadrightarrow,d] & \Or(2k) \arrow[twoheadrightarrow,r,"\det"] \arrow[twoheadrightarrow,d] & \{\pm 1\} \arrow[equal,d] \\
\PSO(2k) \arrow[r] & \PO(2k) \arrow[r,"\det"] & \{\pm 1\}
\end{tikzcd}
\end{equation}
\begin{equation}\label{cd:Oodd}
\begin{tikzcd}[column sep=large]
\{I_{2k-1}\} \arrow[hookrightarrow,r] \arrow[hookrightarrow,d] & \{\pm I_{2k-1}\}  \arrow[twoheadrightarrow,r] \arrow[hookrightarrow,d] & \{\pm 1\} \arrow[equal,d] \\
\SO(2k-1) \arrow[hookrightarrow,r] \arrow[twoheadrightarrow,d] & \Or(2k-1) \arrow[twoheadrightarrow,r,"\det"] \arrow[twoheadrightarrow,d] & \{\pm 1\} \arrow[d]\\
\PSO(2k-1) \arrow[r] & \PO(2k-1) \arrow[r,"\det"] & \{ 1\}
\end{tikzcd}
\end{equation}

In both \eqref{cd:Oeven} and \eqref{cd:Oodd}, each column is exact, and the top two rows are exact. By the nine-lemma, the bottom rows are exact as well. In particular, the quotient maps induce isomorphisms on homotopy groups in degrees $\geq 2$, and we obtain:
\begin{align*}
\pi_{i}(\PO(2k))&\iso \pi_{i}(\PSO(2k)) \iso \pi_{i}(\SO(2k))\quad \text{for all} \quad i\geq 2, \quad \text{and}\\
\PO(2k-1)&\iso \PSO(2k-1)\iso \SO(2k-1)
\end{align*}

Thus the only remaining case is the computation of the fundamental group of $\PSO(2k)$ for $k\ge 2$ (equivalently, for $n\ge 3$ even). Recall that the spin group $\Spin(n)$ is the universal double cover of $\SO(n)$ \cite[Page 74]{MiToToG1991}. Since $\SO(n)\to \PSO(n)$ is also a double cover for $n\ge 3$, the group $\Spin(n)$ is in fact the universal covering group of $\PSO(n)$ as well. Let
\[
1\longrightarrow \Ker(p)\longrightarrow \Spin(n)\xrightarrow{\,p\,}\PSO(n)\longrightarrow 1
\]
denote the corresponding covering homomorphism. Then $\pi_{1}(\PSO(n))\iso \Ker(p)$, and because $\Spin(n)$ is simply connected, $\Ker(p)$ agrees with the center $Z(\Spin(n))$. By \cite[Theorem II.4.4]{MiToToG1991}, for $n\ge 3$ the center is
\begin{equation}\label{ZSpinn}
Z(\Spin(n))\cong
\begin{cases}
\ZZ/2 & \text{if $n$ is odd},\\
\ZZ/2\oplus\ZZ/2 & \text{if $n\equiv 0 \pmod 4$},\\
\ZZ/4 & \text{if $n\equiv 2 \pmod 4$}.
\end{cases}
\end{equation}

\subsection{First unstable homotopy group of \texorpdfstring{$\fO{n}$}{On}}

The standard inclusion $i:\Or(n) \hookrightarrow \Or(n+1)$ is $(n-1)$-connected, hence it induces an isomorphism on homotopy groups in degrees less than $n-1$ and an epimorphism in degree $n-1$. One way to see this is to use the fibration $\Or(n) \hookrightarrow \Or(n+1) \rightarrow \Or(n+1)/\Or(n)\simeq S^{n}$, and the associated long exact sequence in homotopy.

In particular, the first unstable homotopy group of $\Or(n)$ occurs in degree $n-1$.  The relevant portion of the long exact sequence is
\begin{equation*}\label{segmentles}
\begin{tikzcd}
\pi_{n}(S^{n}) \arrow[r,"\partial"] & \pi_{n-1}(\Or(n)) \arrow[r,"i_{*}"] & \pi_{n-1}(\Or(n+1))  \arrow[r]  & 0.
\end{tikzcd}
\end{equation*}

By exactness of the sequence above, there is a short exact sequence
\begin{equation}\label{ses:piOn}
\begin{tikzcd}
0 \arrow[r] & \Ker i_{*} \arrow[r,hookrightarrow] & \pi_{n-1}(\Or(n)) \arrow[r,"i_{*}"] & \pi_{n-1}(\Or)  \arrow[r]  & 0,
\end{tikzcd}
\end{equation}
where $\Or\coloneqq \colim_{n\to \infty} \Or(n)$ and $\pi_{n-1}(\Or)\iso \pi_{n-1}(\Or(n+1))$.
When $n=2$, we have $\Or(2)\simeq S^{1}$, hence $\pi_{1}(\Or(2))=\pi_{1}(S^{1})=\ZZ$. Now suppose $n=3$ or $n=7$. By \cite[Corollary IV.10.6, Theorem IV.10.8]{WhiHT1978}, the group $\Ker i_{*}$ is trivial. Therefore $\pi_{n-1}(\Or(n))\iso\pi_{n-1}(\Or)$, and in particular $\pi_{n-1}(\Or(n))$ is trivial in these cases.

Finally, if $n\neq 1,2,3,7$, the short exact sequence \eqref{ses:piOn} splits. Moreover, by \cite[Corollary IV.6.14]{MiToToG1991} we have
\begin{align}\label{pin1On}
\pi_{n-1}(\Or(n)) \iso %
\begin{cases}
\ZZ \oplus \ZZ & \text{if $n\equiv0,4$ (mod 8)}, \\
\ZZ/2 \oplus \ZZ/2& \text{if $n\equiv1$ (mod 8)}, \\
\ZZ \oplus \ZZ/2 & \text{if $n\equiv2$ (mod 8)}, \\
\ZZ/2 & \text{if $n\equiv3,5,7$ (mod 8)},\\
\ZZ & \text{if $n\equiv6$ (mod 8)}.
\end{cases}
\end{align}

Let $G\in \Bigl\{\Or(n), \SO(n), \PO(n), \PSO(n)\Bigr\}$. Tables \ref{tab:connectcomp}, \ref{tab:fundgroup}, \ref{tab:lowhg}, and \ref{tab:1stunstable} summarize the previous results.

\begin{table}[h!]
\centering
{\small \renewcommand{\arraystretch}{1.2}
\begin{tabular}{|c|cccc|}
\hline
$G$ & $\Or(n)$ & $\SO(n)$ & $\PO(2k)$ & $\PSO(2k)$\\
\hline
$\pi_{0}(G)$ & $\ZZ/2$ & 1 & $\ZZ/2$ & 1\\
\hline
\end{tabular}}
\caption{Connected components of compact Lie groups related to the complex orthogonal group for $n\geq 1$ and $k\geq 2$.}
\label{tab:connectcomp}
\end{table}

\begin{table}[h!]
\centering
{\small \renewcommand{\arraystretch}{1.2}
\begin{tabular}{|c|cccc|}
\hline
$G$ & $\Or(n)$ & $\SO(n)$ & $\PO(2k)$ & $\PSO(2k)$\\
\hline
$\pi_{1}(G)$ & $\ZZ/2$ & $\ZZ/2$ & $Z(\Spin(2k))$ & $Z(\Spin(2k))$\\
\hline
\end{tabular}}
\caption{Fundamental group of compact Lie groups related to the complex orthogonal group for $n\geq 3$ and $k\geq 2$.}
\label{tab:fundgroup}
\end{table}

\begin{table}[h!]
\centering
{\small \renewcommand{\arraystretch}{1.2}
\begin{tabular}{|c|cccccccc|}
\hline
$i>1$ and $i$ (mod $8$) & $0$ & $1$ & $2$ & $3$ & $4$ & $5$ & $6$ & $7$\\
\hline
$\pi_{i}(G)$ & $\ZZ/2$ & $\ZZ/2$ & $0$ & $\ZZ$ & $0$ & $0$ & $0$ & $\ZZ$\\
\hline
\end{tabular}}
\caption{Homotopy groups of compact Lie groups related to the complex orthogonal group for for $n\geq 3$ and $i=2,\dots,n-2$}
\label{tab:lowhg}
\end{table}

\begin{table}[h!]
\centering
{\small \renewcommand{\arraystretch}{1.2}
\begin{tabular}{|c|cccccccc|}
\hline
$n$ equals to & & & $2$ & $3$ & $4$ & $5$ & $6$ & $7$\\
\hline
$\pi_{n-1}(G)$ &  &  & $\ZZ$ & $0$ & $\ZZ\oplus \ZZ$ & $\ZZ/2$ & $\ZZ$ & $0$\\
\hline
\hline
 $n>7$ and $n$ (mod $8$) & $0$ & $1$ & $2$ & $3$ & $4$ & $5$ & $6$ & $7$\\
\hline
$\pi_{n-1}(G)$ & $\ZZ \oplus \ZZ$  & $\ZZ/2\oplus \ZZ/2$ & $\ZZ/2\oplus \ZZ$ & $\ZZ/2$ & $\ZZ\oplus \ZZ$ & $\ZZ/2$ & $\ZZ$ & $\ZZ/2$\\
\hline
\end{tabular}}

\caption{First unstable homotopy group of compact Lie groups related to the complex orthogonal group.}
\label{tab:1stunstable}
\end{table}

\subsection{Stabilization}
Let $m,n \in \NN$ and $m\leq n$. Define the map $\s:\Or(m) \rightarrow \Or(m+n)$ by
\begin{equation*}
\s(A)=
\begin{pmatrix}
A & 0\\
0 & I_{n}
\end{pmatrix}.
\end{equation*}

Since the map $\s$ is equal to the composite of consecutive canonical inclusions, it follows that $\s$ is $(m-1)$-connected.

\begin{notation}
Let $\esta$ denote the homomorphisms the map $\s$ induces on homotopy groups. From now on, we will identify $\pi_{i}(\Or(m))$ with $\pi_{i}(\Or(m+n))$ for all $i<m-1$ through the isomorphism
\begin{equation}
\begin{tikzcd}[row sep=tiny, column sep=normal]
\esta:\pi_{i}(\Or(m)) \arrow[r,rightarrow,"\iso"] & \pi_{i}(\Or(m+n)).
\end{tikzcd}
\end{equation}
\end{notation}

The following lemma is a technical result that will be used to prove Lemma \ref{L:sjO} and Lemma \ref{L:RLtpO}.
\begin{lemma}\label{techl}
Let $G$ be a Lie group and let $G_{0}$ be the component of the identity. If $c:G\rightarrow G$ is conjugation by $P \in G_{0}$, then there is a basepoint preserving homotopy $H$ from $c$  to $\id_{G}$ such that for all $t \in [0,1]$, $H(-,t)$ is a homomorphism.
\end{lemma}
\begin{proof}
Since $G_{0}$ is path-connected, there exists a path $\alpha$ from $P$ to $e_{G}$ in $G$. Define $H:G\times [0,1] \rightarrow G$ by $H(A,t)=\alpha(t)A\alpha(t)^{-1}$. Observe that $H(-,t):G\to G$, $A\mapsto H(A,t)$ is a homomorphism. Moreover, $H$ is such that  $H(e_{G},t)=e_{G}$, $H(A,0)=c(A)$, and $H(A,1)=A$.
\end{proof}

\begin{lemma}\label{L:sjO}
Let $n,r \in \NN$. For all $j=1,\dots,r$ define $\s_{j}:\Or(n) \rightarrow \Or(rn)$ by 
\[
\s_{j}(A)=\diag(I_{n},\dots,I_{n},A,I_{n},\dots,I_{n}),
\]
where $A$ is in the $j$-th position. Then the map $\s_{j}$ is pointed homotopic to $\s_{j+1}$ for $j=1,\dots,r-1$.
\end{lemma}
\begin{proof}
The result is clear when $n=1$ or $r=1$, so let us assume that $n\geq2$ and $r\geq 2$.

We want to prove that $\s_{j+1}=c\,\circ\, \s_{j}$ for all $j=1,2, \dots, r-1$ where $c: \Or(rn) \to \Or(rn)$ is conjugation by a matrix in the connected component of the identity in $\Or(rn)$.

For $j=1,2, \dots, r-1$, let
\[
T=\begin{pmatrix}
0 & I_n\\
I_n & 0
\end{pmatrix}\in \Or(2n)%
\quad \text{and} \quad%
P_{j}=%
\begin{pmatrix}
I_{(j-1)n} & & \\
 & T & \\
 & & I_{(r-j-1)n}
\end{pmatrix} \in \Or(rn).
\]
Thus $P_{j}$ is the permutation matrix that interchanges the $j$-th and $(j+1)$-st
$n\times n$ diagonal blocks. In particular, for every $A\in \Or(n)$ we have
\[
P_{j}\,\s_{j}(A)\,P_{j}^{-1}=\s_{j+1}(A).
\]
Moreover, $\det(P_{j})=\det(T)=(-1)^n$.

If $n$ is even, then $\det(P_{j})=1$, so $P_{j}$ lies in the connected component of the identity in $\Or(rn)$. Therefore $\s_{j+1}$ is obtained from $\s_{j}$ by conjugation by $P_{j}\in \SO(rn)$, and the result follows from Lemma \ref{techl}.

%

Assume now that $n$ is odd, so $\det(P_{j})=-1$. Consider the block diagonal matrix
\[
D=\diag(-I_{n},\,I_{n},\,\dots,\,I_{n})\in \Or(rn),
\]
which satisfies $\det(D)=\det(-I_{n})=(-1)^n=-1$. For $C_{j}\coloneqq DP_{j}$ we have that
\[
\det(C_{j})=\det(D)\det(P_{j})=(-1)\cdot(-1)=1.
\]
Hence $C_{j}$ is in the identity component of $\Or(rn)$. Moreover,
\[
C_{j}\,\s_{j}(A)\,C_{j}^{-1}%
= D\bigl(P_{j}\,\s_{j}(A)\,P_{j}^{-1}\bigr)D^{-1}%
= D\,\s_{j+1}(A)\,D%
= \s_{j+1}(A).
\]
Thus $\s_{j+1}$ is obtained from $\s_{j}$ by conjugation by $C_{j}\in \SO(rn)$, and the result follows from Lemma \ref{techl}.
\end{proof}

\begin{notation}\label{Notationesta}
We call the $\s_{j}$ maps \textit{stabilization maps}. Since the first stabilization map $\s_{1}$ is equal to $\s:\Or(n) \to \Or(n+(r-1)n)$, it follows that each stabilization map $\s_{j}$ is $(n-1)$-connected for all $j=1,\dots,r$. By Lemma \ref{L:sjO}, the induced homomorphisms on homotopy groups by the componentwise stabilization maps are equal. Hence, we use $\esta$ to denote the common map
\[
\esta = \pi_{i}(\s_{1}) = \cdots = \pi_{i}(\s_{r}). 
\]
Thus, we identify $\pi_{i}(\Or(n))$ with $\pi_{i}(\Or(rn))$ for $i<n-1$ via $\esta$. This identification justifies a slight abuse of notation: we write $x = \esta(x)$ for $x\in \pi_{i}(\Or(n))$ and $i<n-1$.
\end{notation}

\subsection{Operations on homotopy groups}

We do not write proofs of some results in this section given that they are similar to the proofs of the results in \cite{TAA2021}.

\begin{proposition}\label{P:dsO}
\cite[Proposition 2.5]{TAA2021}
Let $i\in \NN$. The homomorphism $\oplus_{*}: \pi_{i}(\Or(m))\times \pi_{i}(\Or(n)) \rightarrow \pi_{i}(\Or(m+n))$ is equal to 
\[
\oplus_{*}(x,y)=\esta(x)+\esta(y) \quad\text{for}\quad x\in \pi_{i}(\Or(m)) \quad\text{and}\quad y\in\pi_{i}(\Or(n)).
\]
\end{proposition}

\begin{corollary}\label{C:dsO}
\cite[Corollary 2.6]{TAA2021}
If $m<n$ and $i<m-1$, then $\oplus_{*}(x,y)=x+y$ for $x\in \pi_{i}(\Or(m))$ and $y\in\pi_{i}(\Or(n))$.
\end{corollary}

\begin{proposition}\label{P:rdsO}
\cite[Proposition 2.7]{TAA2021}
Let $i\in \NN$. The homomorphism $\oplus^{r}_{*}:\pi_{i}(\Or(n))\rightarrow \pi_{i}(\Or(rn))$ is equal to
\[
\oplus^{r}_{*}(x)=r\esta(x) \quad\text{for}\quad x\in \pi_{i}(\Or(n)).
\]
\end{proposition}

\begin{corollary}\label{C:rdsO}
\cite[Corollary 2.8]{TAA2021}
If $i<n-1$, then $\oplus_{*}^{r}(x)=rx$ for $x\in \pi_{i}(\Or(n))$.
\end{corollary} 

\begin{lemma}\label{L:RLtpO}
Let $L, R: \Or(m)\rightarrow \Or(mn)$ be the homomorphisms $L(A)=A\tensor I_{n}$ and $R(A)=I_{n}\tensor A$. There is a basepoint preserving homotopy from $L$  to $R$. 
\end{lemma}
\begin{proof}
If $m=1$ or $n=1$, then $L=R$ and there is nothing to prove. Hence assume
$m,n\ge 2$.  Let $A=(a_{ij})\in \Or(m)$.  Then
\[
L(A)=
\begin{pmatrix}
a_{11}I_n & \cdots & a_{1m}I_n\\
\vdots    & \ddots & \vdots\\
a_{m1}I_n & \cdots & a_{mm}I_n
\end{pmatrix}%
\quad \text{and} \quad
R(A)=
\begin{pmatrix}
A&  \cdots &  0\\
\vdots&  \ddots &  \vdots\\
0&  \cdots &  A
\end{pmatrix}=A^{\oplus n}= \s_{1}(A)\s_{2}(A)\cdots \s_{n}(A),
\]
where $\s_{j}$ are the componentwise stabilization maps from Lemma \ref{L:sjO} for all $1\leq j\leq n$. 

Let $P_{m,n}$ be the permutation matrix
\begin{align*}
P_{m,n}=[&e_{1},e_{n+1},e_{2n+1},e_{3n+1},\dots, e_{(m-1)n+1},\\
		      &e_{2},e_{n+2},e_{2n+2},e_{3n+2},\dots, e_{(m-1)n+2},\\
		      &\dots,\\
		      &e_{n-1},e_{2n-1},e_{3n-1},e_{4n-1},\dots,e_{mn-1},\\ 
		      &e_{n},e_{2n},e_{3n},e_{4n},\dots, e_{(m-1)n},e_{mn}]
\end{align*}
where $e_{i}$ is the $i$-th standard basis vector of $\CC^{mn}$ written as a column vector.
Then $P_{m,n}$ is orthogonal and satisfies the standard identity $P_{m,n}\,(I_{n}\otimes A)\,P_{m,n}^{-1}=A\otimes I_{n}$, so in particular
\[
L(A)=P_{m,n}\,R(A)\,P_{m,n}^{-1}.
\]

If $\det(P_{m,n})=1$, the result follows from Lemma \ref{techl}. Assume now that $\det(P_{m,n})=-1$. 
Let $D,U\in \Or(rn)$ be the permutation matrices
\[
D=
\begin{pmatrix}
I_{rn-2} & & \\
 & 0 & 1 \\
 & 1 & 0
\end{pmatrix}%
\quad \text{and} \quad%
U=
\begin{pmatrix}
0 & 1 &  \\
1 & 0 &  \\
  & & I_{rn-2}
\end{pmatrix}.
\]
Then $\det(D)=\det(U)=-1$, and hence $\det(P_{m,n}D)=\det(P_{m,n}U)=\det(D^{-1}U)=1$, so $P_{m,n}D$ and $P_{m,n}U$ lie in the identity component of $\Or(mn)$.

By construction, $U$ acts only on the first two coordinates (in the first block); similarly, $D$ acts only on the last two coordinates (which lie in the $n$-th block).  Thus
\[
R(A)=\Bigl(D \,\s_{1}(A)\cdots\s_{n-1}(A)\, D^{-1}\Bigr)\Bigl(U \s_{n}(A) U^{-1}\Bigr).
\]
Conjugating by $P_{m,n}$ and regrouping gives
\begin{align*}
L(A)&=P_{m,n}\,R(A)\,P_{m,n}^{-1}\\[2mm]
&=P_{m,n}\Bigl(D \,\s_{1}(A)\cdots\s_{n-1}(A)\, D^{-1}\Bigr)\Bigl(U \s_{n}(A) U^{-1}\Bigr)P_{m,n}^{-1}\\[2mm]
&=(P_{m,n}D)\s_{1}(A)\cdots\s_{n-1}(A)(P_{m,n}D)^{-1} (P_{m,n}U)\s_{n}(A)(P_{m,n}U)^{-1}.
\end{align*}
Lemma \ref{techl} provides pointed homotopies from each corresponding conjugation to the identity. Applying these to the two factors and using continuity of multiplication yields a basepoint-preserving homotopy from $L$ to $R$.
\end{proof}

\begin{proposition}\label{P:tpO}
\cite[Proposition 2.10]{TAA2021}
Let $i\in \NN$, the homomorphism $\otimes_{*}:\pi_{i}(\Or(m))\times\pi_{i}(\Or(n)) \to \pi_{i}(\Or(mn))$ is given by 
\[
\otimes_{*}(x,y)=n\esta(x)+m\esta(y) \quad\text{for}\quad x \in \pi_{i}(\Or(m))  \quad\text{and}\quad y\in \pi_{i}(\Or(n)).
\]
\end{proposition}

\begin{corollary}\label{C:tpO}
\cite[Corollary 2.11]{TAA2021}
If $m<n$ and $i<m-1$, then $\otimes_{*}(x,y)=nx+my$ for $x \in \pi_{i}(\Or(m))$ and $y\in \pi_{i}(\Or(n))$.
\end{corollary}

\begin{remark}
Under the hypothesis of Corollary \ref{C:tpO}, the homomorphism 
\[
\tensor_{*}:\pi_{m-1}(\Or(m))\times\pi_{m-1}(\Or(n)) \to \pi_{m-1}(\Or(mn))
\]
is given by $\tensor_{*}(x,y)=n\esta(x)+m\esta(y)= n\esta(x)+my$. This can be seen by observing that $\otimes_{*}$ is equal to the sum of the two paths around diagram \eqref{cd:stabtensor}.

\begin{equation}\label{cd:stabtensor}
\begin{tikzcd}[row sep=large,column sep=scriptsize]
&\pi_{m-1}(\Or(m))\times\pi_{m-1}(\Or(n)) \arrow[dr,twoheadrightarrow,"\proj_{2}"] & \\
\pi_{m-1}(\Or(m)) \arrow[ru,twoheadleftarrow,"\proj_{1}"] \arrow[d,twoheadrightarrow,"\esta"] & &\pi_{m-1}(\Or(n)) \arrow[d,rightarrow,"\esta"]\\
\pi_{m-1}(\Or(mn)) & & \pi_{m-1}(\Or(mn)) \arrow[dl,rightarrow,"\times m"] \arrow[u,leftarrow,"\iso"]\\
& \pi_{m-1}(\Or(mn))\times \pi_{m-1}(\Or(mn)) \arrow[ul,leftarrow,"\times n"] \arrow[d,rightarrow,"+"]  &\\
&\pi_{m-1}(\Or(mn))& 
\end{tikzcd}
\end{equation}
\end{remark}

\begin{proposition}\label{P:rtpO}
\cite[Proposition 2.12]{TAA2021}
Let $i\in \NN$. The homomorphism $\otimes^{r}_{*}:\pi_{i}(\Or(n)) \rightarrow \pi_{i}(\Or(n^{r}))$ is given by 
\[
\otimes^{r}_{*}(x)=rn^{r-1}\esta(x) \quad\text{for}\quad x \in \pi_{i}(\Or(n)).
\]
\end{proposition}

\begin{corollary}\label{C:rtpO}
\cite[Corollary 2.13]{TAA2021}
If $i<n-1$, the map $\otimes^{r}_{*}(x)=rn^{r-1}x$ for $x \in \pi_{i}(\Or(n))$.
\end{corollary}

\subsection{Tensor product on the quotient}

We want to describe the effect of the tensor product operation on the homotopy groups of the projective complex orthogonal group. 

The methods used in \cite{TAA2021} to establish decomposition of topological Azumaya algebras apply to those whose degrees are relatively prime. For this reason, we study the tensor product 
\begin{equation}\label{qO}
\begin{tikzcd}
\otimes:\PO(m)\times \PO(n) \arrow[r] & \PO(mn)
\end{tikzcd}
\end{equation}
in two cases: when $m$ and $n$ are odd, and when $m$ is even and $n$ is odd.

\subsubsection{Case $m$ and $n$ odd}
Let $m$ and $n$ be positive integers such that $m$ and $n$ are odd. Since $\PO(m)\times\PO(n)=\SO(m)\times\SO(n)$, the tensor product in \eqref{qO} may be written as 
\begin{equation}\label{qOo}
\begin{tikzcd}
\otimes:\SO(m)\times \SO(n) \arrow[r] & \SO(mn),
\end{tikzcd}
\end{equation}
and by Proposition \ref{P:tpO} we have

\begin{proposition}\label{P:tppiiOo}
Let $i\in \NN$. The homomorphism 
\begin{equation*}
\begin{tikzcd}
\otimes_{*}: \pi_{i}(\SO(m))\times\pi_{i}(\SO(n)) \arrow[r] & \pi_{i}(\SO(mn))
\end{tikzcd}
\end{equation*}
is given by $\otimes_{*}(x,y)=n\esta(x)+m\esta(y)$
for $x \in \pi_{i}(\SO(m))$ and $y\in \pi_{i}(\SO(n))$.
\end{proposition}

\subsubsection{Case $m$ even and $n$ odd}
Let $m$ and $n$ be positive integers such that $m$ is even and $n$ is odd. The tensor product operation $\otimes:\Or(m)\times\SO(n) \rightarrow \Or(mn)$ sends the center of $\Or(m)\times\SO(n)$ to the center of $\Or(mn)$. As a consequence, the operation descends to the quotient
\begin{equation}\label{qOeo}
\begin{tikzcd}
\otimes:\PO(m)\times \SO(n) \arrow[r] & \PO(mn).
\end{tikzcd}
\end{equation} 

\begin{proposition}\label{P:tppiiOeo}
Let $i>1$. The homomorphism 
\begin{equation*}
\begin{tikzcd}
\otimes_{*}: \pi_{i}(\PO(m))\times\pi_{i}(\SO(n))  \arrow[r] &  \pi_{i}(\PO(mn))
\end{tikzcd}
\end{equation*} 
is given by $\otimes_{*}(x,y)=n\esta(x)+m\esta(y)$ for $x \in \pi_{i}(\PO(m))$ and $y\in \pi_{i}(\SO(n))$.
\end{proposition}
\begin{proof}
There is a map of fibrations
\begin{equation}\label{cd:fOeo}
\begin{tikzcd}
Z(\Or(m))\times\{I_{n}\} \arrow[r,hookrightarrow] \arrow[d,"\otimes"] & \Or(m)\times \SO(n) \arrow[r,twoheadrightarrow] \arrow[d,"\otimes"] & \PO(m)\times\SO(n) \arrow[d,"\otimes"]\\
Z(\Or(mn)) \arrow[r,hookrightarrow] & \Or(mn) \arrow[r,twoheadrightarrow] & \PO(mn).
\end{tikzcd}
\end{equation}
From the homomorphism between the long exact sequences associated to the fibrations in diagram \eqref{cd:fOeo}, we obtain a commutative square
\begin{equation*}
\begin{tikzcd}
\pi_{i}(\Or(m))\times\pi_{i}(\SO(n)) \arrow[r,"\iso"] \arrow[d,"\otimes_{*}"] & \pi_{i}(\PO(m))\times\pi_{i}(\SO(n)) \arrow[d,"\otimes_{*}"]\\
\pi_{i}(\Or(mn)) \arrow[r,"\iso"] & \pi_{i}(\PO(mn)),
\end{tikzcd}
\end{equation*}
for $i>1$. From this diagram and Proposition \ref{P:tpO} we have that for all $i>1$, $\otimes_{*}(x,y)=n\esta(x)+m\esta(y)$ for $x \in \pi_{i}(\PO(m))$ and $y\in \pi_{i}(\SO(n))$.
\end{proof}

\begin{proposition}\label{P:tppi1Oeo}
Let $m$ be an even positive integer and $n$ an odd positive integer. The homomorphism induced on fundamental groups
\[
\otimes_{*}:\pi_{1}(\PO(m))\times\pi_{1}(\SO(n)) \longrightarrow \pi_{1}(\PO(mn))
\]
is described as follows.
\begin{enumerate}
\item If $m=2$ and $n=1$, then $\tensor_{*}(x,1)=x$ for $x\in \ZZ$.

\item If $m=2$ and $n\geq 3$, then $\tensor_{*}(x,y)=\pm x$ for $x\in \ZZ,\ y\in \ZZ/2$.

\item If $m\equiv0$ (mod $4$), then $\otimes_{*}(x,y)=x$ for $x\in \ZZ/2\oplus\ZZ/2$ and $y\in \ZZ/2$.

\item If $m>2$ and $m\equiv2$ (mod $4$), then $\otimes_{*}(x, y)=\pm x$ for $x \in \ZZ/4$ and $y\in \ZZ/2$.
\end{enumerate}
\end{proposition}
\begin{proof}
From the long exact sequence associated to diagram \eqref{cd:fOeo} there is a diagram of exact sequences
\begin{equation}\label{cd:meno}
\begin{tikzcd}
0 \arrow[d] & 0 \arrow[d] \\
\pi_{1}(\Or(m))\times \pi_{1}(\SO(n)) \arrow[d] \arrow[r,"\otimes_{*}"]  & \pi_{1}(\Or(mn)) \arrow[d] \\
\pi_{1}(\PO(m))\times \pi_{1}(\SO(n)) \arrow[d] \arrow[r,"\otimes_{*}"] & \pi_{1}(\PO(mn)) \arrow[d]\\
\pi_{0}Z(\Or(m)) \arrow[d] \arrow[r,equal] & \pi_{0}Z(\Or(mn)) \arrow[d]\\
\pi_{0}(\Or(m)) \arrow[d] \arrow[r,"\otimes_{*}"] & \pi_{0}(\Or(mn)) \arrow[d]\\
\pi_{0}(\PO(m)) \arrow[r,"\otimes_{*}"] & \pi_{0}(\PO(mn)).
\end{tikzcd}
\end{equation}
Since $Z(\Or(m))$ is contained in the connected component of the identity of $\Or(m)$, then the homomorphism $\pi_{0}Z(\Or(m)) \to \pi_{0}(\Or(m))$ is trivial. Hence we have a diagram of short exact sequences
\begin{equation}\label{cd:meno2}
\begin{tikzcd}[column sep=large]
\pi_{1}(\Or(m))\times \pi_{1}(\SO(n)) \arrow[r,hookrightarrow] \arrow[d,"\otimes_{*}"] & \pi_{1}(\PO(m))\times \pi_{1}(\SO(n)) \arrow[r,twoheadrightarrow] \arrow[d,"\otimes_{*}"] & \pi_{0}Z(\Or(m)) \arrow[d,equal]\\
\pi_{1}(\Or(mn)) \arrow[r,hookrightarrow] & \pi_{1}(\PO(mn)) \arrow[r,twoheadrightarrow] & \pi_{0}Z(\Or(mn)).
\end{tikzcd}
\end{equation}

We consider the following cases.

\noindent \textbf{CASE 1.} Let $m=2$.

\textbf{Subcase i.} Let $n=1$. Observe that, under these assumptions, the tensor product map $\tensor: \Or(2)\times \SO(1) \to \Or(2)$ is given by $\tensor(A,1)=A$ for all $A\in\Or(2)$. It follows that the left vertical map in diagram \eqref{cd:meno2} is the projection onto the first factor. Hence \eqref{cd:meno2} can be written as
\begin{equation}\label{cd:m2n1}
\begin{tikzcd}[column sep=large]
0 \arrow[r] & \ZZ\times 1 \arrow[r] \arrow[d,"\proj_{1}"] & \ZZ\times 1 \arrow[r] \arrow[d,"\otimes_{*}"] & \ZZ/2 \arrow[r] \arrow[equal,d] & 0\\
0 \arrow[r] & \ZZ \arrow[r] & \ZZ \arrow[r] & \ZZ/2 \arrow[r] & 0.
\end{tikzcd}
\end{equation}
Moreover, the top and bottom rows are necessarily isomorphic to the non-split short exact sequence $0\to \ZZ \xrightarrow{\times 2} \ZZ \to \ZZ/2 \to 0$. Therefore the induced homomorphism on fundamental groups, $\tensor_{*}: \pi_{1}(\PO(2))\times \pi_{1}(\SO(1)) \to \pi_{1}(\PO(2))$ is given by $\tensor_{*}(x,1)=x$ for all $x\in \ZZ$.

\textbf{Subcase ii.} Let $n\geq 3$. By Proposition \ref{P:tpO}, the left vertical map in diagram \eqref{cd:meno2} is given by $\tensor_{*}(x,y)=n\esta(x)+2\esta(y)$ where $\esta:\pi_{1}(\Or(2)) \to \pi_{1}(\Or(2n))$ is surjective by the connectivity of the componentwise stabilization maps (Notation \ref{Notationesta}). Since $\pi_{1}(\Or(2))\cong \ZZ$ and $\pi_{1}(\Or(2n))\cong \ZZ/2$ for $n\ge 3$, we have $2\,\esta(y)=0$ and $n\,\esta(x)=\esta(x)$ in $\ZZ/2$. It follows that $\tensor_{*}(x,y)=\esta(x)=q(\proj_{1}(x,y))$ where $q:\ZZ\to \ZZ/2$ is the canonical projection. 

Let $n=2k+1$ for $k\geq 1$. Then $Z(\Spin(2n))=Z(\Spin(4k+2))\iso \ZZ/4$. Thus, \eqref{cd:meno2} takes the form
\begin{equation}\label{cd:m2n3}
\begin{tikzcd}[column sep=large]
0 \arrow[r] & \ZZ\times \ZZ/2\arrow[r,"i"] \arrow[d,"q\,\circ\, \proj_{1}"] & \ZZ\times  \ZZ/2 \arrow[r,"p"] \arrow[d,"\otimes_{*}"] & \ZZ/2 \arrow[r] \arrow[equal,d] & 0\\
0 \arrow[r] & \ZZ/2 \arrow[r,"i'"] & \ZZ/4 \arrow[r,"p'"] & \ZZ/2 \arrow[r] & 0.
\end{tikzcd}
\end{equation}

Out of the $4$ homomorphisms in $\Hom(\ZZ\times \ZZ/2,\ZZ/2)\iso \ZZ/2\times \ZZ/2$, exactly $3$ are surjective. In our situation, injectivity of $i$ and exactness at $\ZZ\times \ZZ/2$ of the top row forces the surjection $p$ to factor through the projection to the first coordinate, namely $p(x,y)=q(\proj_{1}(x,y))$, and hence the injection $i$ is necessarily $i(x,y)=(2x,y)$ for all $(x,y) \in \ZZ\times\ZZ/2$. Furthermore, we notice that the bottom row is the unique non-split short exact sequence given by $i'(u)= 2u$ and $p'(v)=v$ for all $u\in \ZZ/2$ and $v\in \ZZ/4$.

Using the identifications above, and a careful analysis of the commutativity constraints in \eqref{cd:m2n3}, one finds that the induced tensor product homomorphism $\tensor_{*}: \pi_{1}(\PO(2))\times \pi_{1}(\SO(n)) \longrightarrow \pi_{1}(\PO(2n))$ 
is such that $\tensor_{*}(1,0)=r$ with $r\in\{1,3\}$ and $\tensor_{*}(0,1)=0$. Therefore, $\tensor_{*}(x,y) = rx = \pm x$ for all $(x,y)\in \ZZ\times \ZZ/2$.

\noindent \textbf{CASE 2.} Let $m>2$.
By Corollary \ref{C:tpO}, and the parities of $m$ and $n$, the homomorphism $\tensor_{*}:\pi_{1}(\Or(m))\times \pi_{1}(\SO(n)) \to \pi_{1}(\Or(mn))$ is the projection onto the first coordinate. Diagram \eqref{cd:meno2} becomes 
\begin{equation}\label{cd:tppi1}
\begin{tikzcd}[column sep=large]
0 \arrow[r] & \ZZ/2\times \ZZ/2 \arrow[r] \arrow[d,"\proj_{1}"] & Z(\Spin(m))\times \ZZ/2 \arrow[r] \arrow[d,"\otimes_{*}"] & \ZZ/2 \arrow[r] \arrow[equal,d] & 0\\
0 \arrow[r] & \ZZ/2 \arrow[r] & Z(\Spin(mn)) \arrow[r] & \ZZ/2 \arrow[r] & 0,
\end{tikzcd}
\end{equation}

\textbf{Subcase i.} Let $m=4k$ with $k\geq 1$, then diagram \eqref{cd:tppi1} takes the form
\begin{equation}\label{cd:tppi14k}
\begin{tikzcd}[column sep=large]
0 \arrow[r] & \ZZ/2\times \ZZ/2 \arrow[r,"i"] \arrow[d,"\proj_{1}"] & (\ZZ/2\oplus \ZZ/2)\times \ZZ/2 \arrow[r,"p"] \arrow[d,"\otimes_{*}"] \arrow[r,dotted,"s",bend left=35,leftarrow] & \ZZ/2 \arrow[r] \arrow[equal,d] & 0\\
0 \arrow[r] & \ZZ/2 \arrow[r,"i'"] & \ZZ/2\oplus \ZZ/2 \arrow[r] & \ZZ/2 \arrow[r] & 0.
\end{tikzcd}
\end{equation}

Let $s: \ZZ/2 \rightarrow (\ZZ/2\oplus \ZZ/2)\times \ZZ/2$ be a section of $p$ such that $s(1)=(0\oplus1,0)$. The composite $\tensor_{*}\circ s$, and $s$ make the short exact sequences in diagram \eqref{cd:tppi14k} split compatibly. Thus
\begin{equation*}
\begin{tikzcd}[column sep=large]
(\ZZ/2\oplus \ZZ/2)\times \ZZ/2 \arrow[r,"\iso"] \arrow[d,"\otimes_{*}"] & \im i\oplus \im s \arrow[d,"\proj_{1}\;\oplus\;\id"]\\
\ZZ/2\oplus \ZZ/2 \arrow[r,"\iso"] & \im i'\oplus \im(\tensor_{*}\circ s).
\end{tikzcd}
\end{equation*}

Hence $\tensor_{*}(\alpha\oplus\beta,y)=\alpha\oplus\beta$ for all $\alpha, \beta, y \in \ZZ/2$.

\textbf{Subcase ii.} 
Let $m=4k+2$ with $k\geq 1$, then diagram \eqref{cd:tppi1} takes the form
\begin{equation}\label{cd:tppi14k2}
\begin{tikzcd}
0 \arrow[r] & \ZZ/2\times \ZZ/2 \arrow[r,"i"] \arrow[d,"\proj_{1}"] & \ZZ/4 \times \ZZ/2 \arrow[r,"p"] \arrow[d,"\tensor_{*}"] & \ZZ/2 \arrow[r] \arrow[equal,d] & 0\\
0 \arrow[r] & \ZZ/2 \arrow[r,"i'"] & \ZZ/4 \arrow[r,"p'"] & \ZZ/2 \arrow[r] & 0,
\end{tikzcd}
\end{equation}

Using an analogous argument to the one in Case~1, Subcase~ii, we can show that $\tensor_{*}(x,y)= \pm x$ for all $(x,y) \in \ZZ/4\times \ZZ/2$.
\end{proof}

\begin{proposition}\label{P:tppi0Oeo}
Let $n$ be a positive odd integer, then the homomorphism  
\begin{equation*}
\begin{tikzcd}
\otimes_{*}: \pi_{0}(\PO(m))\times\pi_{0}(\SO(n)) \arrow[r] & \pi_{0}(\PO(mn))
\end{tikzcd}
\end{equation*}
is a bijection.
\end{proposition}

\section{Proof of Theorem \ref{mainO}} \label{sec:mainO}

Let $m$ and $n$ be positive integers. By applying the classifying-space functor to the homomorphism \eqref{qO} we obtain a map 
\begin{equation*}
\begin{tikzcd}
f_{\tensor}: \B\PO(m)\times\B\PO(n)  \arrow[r] & \B\PO(mn).
\end{tikzcd}
\end{equation*}

\begin{proposition}\label{P:TtildaOeo}
Let $m$ and $n$ be relatively prime positive integers such that $m$ is even, and $n$ is odd. There exists a homomorphism $\widetilde{\Tr}: \PO(m)\times \SO(n) \rightarrow \SO(N)$, for some positive integer $N$, such that the homomorphisms induced on homotopy groups 
\begin{equation*}
\begin{tikzcd}
\widetilde{\Tr}_{i}:\pi_{i}(\PO(m))\times \pi_{i}(\SO(n))  \arrow[r] &  \pi_{i}(\SO(N))
\end{tikzcd}
\end{equation*}
are given by the following expressions. 

\begin{enumerate}
\item If $i=1$, then there exist $a, z, z' \in \ZZ/2$ such that
\begin{equation*}
\begin{cases}
\widetilde{\Tr}_{1}(x,y)=ax+y & \text{if $m=2$},\\
\widetilde{\Tr}_{1}(\alpha\oplus\beta,y)=z\beta+y & \text{if $m\equiv0$ (mod $4$)},\\
\widetilde{\Tr}_{1}(\delta,y)=z'\delta+y & \text{if $m>2$ and $m\equiv2$ (mod $4$)},
\end{cases}
\end{equation*}
for all $x\in \ZZ$, $y, \alpha, \beta \in \ZZ/2$ and $\delta \in \ZZ/4$.

\item Let $d$ denote $\min\{m,n\}$. If $1<i<d-1$,
\[
\widetilde{\Tr}_{i}(x,y)=2umx+vy \quad\text{where}\quad x \in \pi_{i}(\PO(m)), \quad y \in \pi_{i}(\SO(n)),
\]
and $u, v$ are some positive integers, independent of $i$, for which $\bigl|vn-2um^{2}\bigr|=1$.
\end{enumerate}
\end{proposition}
\begin{proof}
We establish the result for $m<n$; the case of $n<m$ follows similarly. Since $m$ and $n$ are relatively prime, there exist positive integers $u$ and $v$ such that $vn-2um^{2}=\pm1$. Let $N$ denote $um^{2}+vn$, and let $\Tr$ denote the composite
\begin{equation*}
\begin{tikzcd}[column sep=large]
\Or(m)\times \SO(n) \arrow[d,"(\tensor^{2}\text{,}\id)"] \\
\SO(m^{2})\times \SO(n) \arrow[d,"(\oplus^{u}\text{,}\oplus^{v})"] \\
\SO(um^{2})\times \SO(vn) \arrow[d,"\oplus"] \\
\SO(N).
\end{tikzcd}
\end{equation*}

Note that the elements $\bigl(\pm I_{m}, I_{n}\bigr)$ are sent to $\bigl(I_{m^{2}},I_{n^{2}}\bigr)$ by $(\tensor^{2},\id)$, hence to the identity by the composite $\Tr$ defined above. Hence $\Tr$ factors through $\PO(m)\times \SO(n)$
\begin{equation*}
\begin{tikzcd}[column sep=huge]
\Or(m)\times \SO(n) \arrow[rd,bend left=20,"\Tr"] \arrow[d,twoheadrightarrow]  &  \\
 \PO(m)\times\SO(n) & \SO(N). \arrow[l,leftarrow,"\widetilde{\Tr}"]
\end{tikzcd}
\end{equation*}

From Corollaries \ref{C:dsO}, \ref{C:rdsO} and \ref{C:rtpO} we have that $\Tr_{i}(x,y)=2umx+vy$ for $i<m-1$. Now, for all $i>1$, the map of fibrations
\begin{equation*}
\begin{tikzcd}
Z(\Or(m))\times\{I_{n}\} \arrow[r,hookrightarrow] \arrow[d] & \Or(m)\times \SO(n) \arrow[r,twoheadrightarrow] \arrow[d,"\Tr"] & \PO(m)\times\SO(n) \arrow[d,"\widetilde{\Tr}"]\\
\{I_{N}\} \arrow[r] & \SO(N) \arrow[r,equal] & \SO(N)
\end{tikzcd}
\end{equation*}
induces a commutative diagram
\begin{equation*}
\begin{tikzcd}
\pi_{i}(\Or(m))\times\pi_{i}(\SO(n)) \arrow[r,"\iso"] \arrow[d,"\Tr_{i}"] & \pi_{i}(\PO(m))\times\pi_{i}(\SO(n)) \arrow[d,"\widetilde{\Tr}_{i}"]\\
\pi_{i}(\SO(N)) \arrow[r,equal] & \pi_{i}(\SO(N)).
\end{tikzcd}
\end{equation*}
This implies that $\widetilde{\Tr}_{i}=\Tr_{i}$ for all $i>1$. In particular, when $1<i<m-1$ we have $\widetilde{\Tr}_{i}(x,y)=\Tr_{i}(x,y)=2umx+vy$ for all $(x,y) \in \pi_{i}(\PO(m))\times\pi_{i}(\SO(n))$.

For $i=1$, the map of fibrations induces the commutative diagram below
\begin{equation}\label{dig:T1OmSOn}
\begin{tikzcd}
\pi_{1}(\Or(m))\times\pi_{1}(\SO(n)) \arrow[r,hookrightarrow] \arrow[d,"\Tr_{1}"] & \pi_{1}(\PO(m))\times\pi_{1}(\SO(n)) \arrow[d,"\widetilde{\Tr}_{1}"]\\
\pi_{1}(\SO(N)) \arrow[r,equal] & \pi_{1}(\SO(N)).
\end{tikzcd}
\end{equation}

\noindent \textbf{CASE 1:} Let $m=2$. Diagram \eqref{dig:T1OmSOn} takes the form
\begin{equation}\label{cd:TTtildem2}
\begin{tikzcd}
\ZZ\times \ZZ/2 \arrow[r,hookrightarrow] & \ZZ\times \ZZ/2 \arrow[d,"\widetilde{\Tr}_{1}"]\\
 & \ZZ/2. \arrow[ul,leftarrow,"\Tr_{1}"]
\end{tikzcd}
\end{equation}
The equality $vn-2um^{2}=\pm1$ implies that $v$ is odd. Hence, $\Tr_{1}(x,y)=2umx+vy=y$, so $\Tr_{1}$ is the projection onto the second coordinate. Moreover, the homomorphism $\pi_{1}(\Or(2)) \to \pi_{1}(\PO(2))$ is multiplication by $\pm 2$. Thus the horizontal homomorphism in diagram \eqref{cd:TTtildem2} is given by $(w,y)\mapsto (\pm2w,y)$ for all $w\in\ZZ$ and $y\in \ZZ/2$. 

The homomorphism $\widetilde{\Tr}_{1}: \ZZ\times\ZZ/2 \to \ZZ/2$ has the form $\widetilde{\Tr}_{1}(x,y)=ax+by$ for all $x\in \ZZ$ and $y\in \ZZ/2$, and for some $a, b \in \ZZ/2$. 

Therefore, $\widetilde{\Tr}_{1}(\pm2w,y)=a(\pm2w)+by=by$ for all $w\in \ZZ$ and $y\in \ZZ/2$. On the other hand, by commutativity of diagram \eqref{cd:TTtildem2}, $\widetilde{\Tr}_{1}(\pm2w,y)=\Tr_{1}(w,y)=y$. Hence $b=1$. Consequently, $\widetilde{\Tr}_{1}(x,y)=ax+y$ for some $a \in \ZZ/2$.


\noindent \textbf{CASE 2:} Let $m>2$.  Diagram \eqref{dig:T1OmSOn} takes the form
\begin{equation}\label{cd:TTtilde}
\begin{tikzcd}
\ZZ/2\times \ZZ/2 \arrow[r,hookrightarrow] & Z(\Spin(m))\times \ZZ/2 \arrow[d,"\widetilde{\Tr}_{1}"]\\
 & \ZZ/2, \arrow[ul,leftarrow,"\Tr_{1}"]
\end{tikzcd}
\end{equation}
where $\Tr_{1}$ is also the projection onto the second coordinate. 

\textbf{Subcase i:} Let $m=4k$ for $k\geq 1$. Diagram \eqref{cd:TTtilde} becomes
\begin{equation*}
\begin{tikzcd}
\ZZ/2\times \ZZ/2 \arrow[r,hookrightarrow] & (\ZZ/2\oplus \ZZ/2)\times \ZZ/2 \arrow[d,"\widetilde{\Tr}_{1}"]\\
 & \ZZ/2 \arrow[ul,leftarrow,"\Tr_{1}"]
\end{tikzcd}
\end{equation*}
where the horizontal homomorphism is the inclusion $(\alpha,\beta)\mapsto (\alpha\oplus0,\beta)$. Thus $\widetilde{\Tr}_{1}(1\oplus0,0)=0$ and $\widetilde{\Tr}_{1}(\mathbf{0},1)=1$. Let $z$ denote $\widetilde{\Tr}_{1}(0\oplus1,0)$. Hence $\widetilde{\Tr}_{1}(\alpha\oplus\beta,y)=z\beta+y$.

\textbf{Subcase ii:} Let $m=4k+2$ for $k\geq 1$. Diagram \eqref{cd:TTtilde} becomes
\begin{equation*}
\begin{tikzcd}
\ZZ/2\times \ZZ/2 \arrow[r,hookrightarrow] & \ZZ/4\times \ZZ/2 \arrow[d,"\widetilde{\Tr}_{1}"]\\
 & \ZZ/2 \arrow[ul,leftarrow,"\Tr_{1}"]
\end{tikzcd}
\end{equation*}
where the horizontal homomorphism is the inclusion $(\alpha,\beta)\mapsto (\alpha,\beta)$. Thus $\widetilde{\Tr}_{1}(0,1)=1$. Let $z'$ denote $\widetilde{\Tr}_{1}(1,0)$. Hence $\widetilde{\Tr}_{1}(\delta,y)=z'\delta+y$.
%
\end{proof}

\subsection{A \texorpdfstring{$d$}{d}-connected map where \texorpdfstring{$d$}{d} is the minimum of two positive integers of opposite parity}

Let $J$ denote the map
\begin{equation*}
\begin{tikzcd}[row sep=tiny]
J:\B\PO(m)\times\B\PO(n) \arrow[r,mapsto] & \B\PO(mn)\times\B\SO(N)\\
(x,y) \arrow[r,mapsto] & \Bigl(f_{\tensor}(x,y)\text{,}\B\widetilde{\Tr}(x,y)\Bigr).
\end{tikzcd}
\end{equation*}

Let $J_{i}$ denote the homomorphism induced on homotopy groups by $J$.
\begin{equation}\label{Ji}
J_{i}: \pi_{i}(\B\PO(m))\times \pi_{i}(\B\PO(n)) \rightarrow \pi_{i}(\B\PO(mn))\times\pi_{i}(\B\SO(N))
\end{equation}

\begin{proposition}\label{P:Jeoim}
Let $m$ and $n$ be relatively prime positive integers such that $m$ is even, and $n$ is odd. Let $d$ denote $\min\{m,n\}$. The homomorphism $J_{i}$ is an isomorphism for all $i<d$.
\end{proposition}
\begin{proof}
We establish the result for $m<n$; the case of $n<m$ follows similarly.

Let $i<m$. Given that the homotopy groups of the spaces involved are zero in degrees $i\equiv 3, 5, 6, 7$ (mod 8), it suffices to prove that $J_{i}$ is an isomorphism for $i=1,2$, and for $i\equiv 0,1,2,4$ (mod 8) with $i>2$.

\noindent \textbf{CASE 1:} Let $i=1$. By Proposition \ref{P:tppi0Oeo} the homomorphism $J_{1}: \ZZ/2 \rightarrow \ZZ/2$ is the identity.

\noindent \textbf{CASE 2:} Let $i=2$. The homomorphism \eqref{Ji} takes the form
\begin{equation*}
\begin{tikzcd}
J_{2}: Z(\Spin(m)) \times \ZZ/2 \arrow[r] & Z(\Spin(mn)) \times \ZZ/2.
\end{tikzcd}
\end{equation*}

\textbf{Subcase i:} Let $m=4k$ for some $k\in \ZZ$. From Propositions \ref{P:tppi1Oeo} and \ref{P:TtildaOeo}, $J_{2}$ is given by
\begin{equation*}
\begin{tikzcd}[row sep=tiny]
J_{2}: (\ZZ/2\oplus \ZZ/2)\times \ZZ/2 \arrow[r] & (\ZZ/2\oplus \ZZ/2) \times \ZZ/2\\
(\alpha\oplus\beta,y) \arrow[r,mapsto] & \Bigl(\alpha\oplus\beta,z\beta+y\Bigr),
\end{tikzcd}
\end{equation*}
i.e. $J_{2}$ is represented by the invertible matrix $\displaystyle{\begin{pmatrix}
1 & 0 & 0\\
0 & 1 & 0 \\
0 & z & 1
\end{pmatrix}}$ for some $z\in \ZZ/2$. Hence $J_{2}$ is an isomorphism.

\textbf{Subcase ii:} Let $m=4k+2$ for some $k\in \ZZ$. From Proposition  \ref{P:tppi1Oeo}, $f_{\tensor}(x,y)=\pm x$.
Thus, by Proposition \ref{P:TtildaOeo}, $J_{2}$ is given by
\begin{equation*}
\begin{tikzcd}[row sep=tiny]
J_{2}: \ZZ/4\times \ZZ/2 \arrow[r] & \ZZ/4 \times \ZZ/2\\
(x,y) \arrow[r,mapsto] & \Bigl(\pm x,z' x+y\Bigr),
\end{tikzcd}
\end{equation*}
i.e. $J_{2}$ is represented by the invertible matrix $\displaystyle{\begin{pmatrix}
\pm 1 & 0 \\
z' & 1
\end{pmatrix}}$ for some $z'\in \ZZ/2$. Hence $J_{2}$ is an isomorphism.

\noindent \textbf{CASE 3:} Let $2<i<m$. The homomorphism \eqref{Ji} takes the form $J_{i}: \ZZ\times\ZZ \rightarrow \ZZ\times\ZZ$ if $i\equiv0,4$ (mod 8), and $J_{i}: \ZZ/2\times \ZZ/2 \rightarrow \ZZ/2\times \ZZ/2$ if $i\equiv1,2$ (mod 8). Note that both homomorphisms are represented by the matrix $\displaystyle{\begin{pmatrix}
n & m\\
2um & v
\end{pmatrix}}$, whose determinant equals $vn-2um^{2}$. Thus, $J_{i}$ is an isomorphism because, by Proposition \ref{P:TtildaOeo}, $vn-2um^{2}=\pm 1$.
\end{proof}

\begin{equation}\label{cd:stabTevenodd}
\begin{tikzcd}[row sep=large,column sep=tiny]
&\pi_{m-1}(\Or(m))\times\pi_{m-1}(\SO(n))  \arrow[dr,rightarrow,"\proj_{2}"] & \\
\pi_{m-1}(\Or(m)) \arrow[ru,leftarrow,"\proj_{1}"] \arrow[d,twoheadrightarrow,"\esta"] & &\pi_{m-1}(\SO(n)) \arrow[dd,rightarrow,equal]\\
\pi_{m-1}(\Or(m^{2})) \arrow[d,rightarrow,"\times (2m)"] & &  \\
\pi_{m-1}(\Or(m^{2})) \arrow[d,rightarrow,"\esta"] & & \pi_{m-1}(\SO(n)) \arrow[d,rightarrow,"\esta"] \\
\pi_{m-1}(\Or(um^{2})) \arrow[u,leftarrow,"\iso"] & &  \pi_{m-1}(\SO(vn)) \arrow[dl,rightarrow,"\times v"] \arrow[u,leftarrow,"\iso"] \\
& \pi_{m-1}(\Or(um^{2}))\times\pi_{m-1}(\SO(vn)) \arrow[d,rightarrow,"\esta\times \esta"] \arrow[ul,leftarrow,"\times u"] & \\
&\pi_{m-1}(\SO(N))\times\pi_{m-1}(\SO(N)) \arrow[d,rightarrow,"+"] \arrow[u,leftarrow,"\iso"]& \\
&\pi_{m-1}(\SO(N))& 
\end{tikzcd}
\end{equation}

\begin{proposition}\label{P:Jeom}
Let $m$ and $n$ be relatively prime positive integers such that $m$ is even, and $n$ is odd. Let $d$ denote $\min\{m,n\}$. The induced homomorphism
\begin{equation}\label{Jepi}
\begin{tikzcd}
J_{d}: \pi_{d}(\B\PO(m))\times \pi_{d}(\B\SO(n)) \arrow[r] & \pi_{d}(\B\PO(mn))\times\pi_{d}(\B\SO(N))
\end{tikzcd}
\end{equation}
is an epimorphism.
\end{proposition}
\begin{proof}
Suppose $m<n$. 

\noindent \textbf{CASE 1:} Let $m=2$. Since $m<n$, we have $n$ an odd integer with $n\geq 3$. By Propositions \ref{P:tppi1Oeo} and \ref{P:TtildaOeo}, the homomorphims $\tensor_{*}:\pi_{1}(\PO(2))\times \pi_{1}(\SO(n)) \to \pi_{1}(\PO(2n))$ and $\widetilde{\Tr}_{1}:\pi_{1}(\PO(2))\times \pi_{1}(\SO(n)) \to \pi_{1}(\SO(N))$ are given by $\tensor_{*}(x,y)=\pm x$ and $\widetilde{\Tr}_{1}(x,y)=ax+y$ for some $a\in \ZZ$, respectively. Thus $J_{2}: \ZZ\times \ZZ/2 \to \ZZ/4\times\ZZ/2$ is given by
\[
J_{2}(x,y)=\bigl(\pm x,ax+y\bigr).
\]
which is an epimorphism.

\noindent \textbf{CASE 2:} Let $m>2$. It follows from Propositions \ref{P:rdsO} and \ref{P:TtildaOeo}, and Corollaries \ref{C:rdsO} and \ref{C:rtpO} that $\Tr_{m-1}:\pi_{m-1}(\Or(m))\times\pi_{m-1}(\SO(n)) \to \pi_{m-1}(\SO(N))$ is given by
\[
\Tr_{m-1}(x,y)=\esta(u\esta(2m\esta(x)))+\esta(v\esta(y))=2um\esta(x)+vy,
\]
where $vn-2um^{2}=\pm1$. This can be seen by observing that $\Tr_{m-1}$ is equal to the sum of the two paths around diagram \eqref{cd:stabTevenodd}. Since $\widetilde{\Tr}_{i}=\Tr_{i}$ for all $i>1$ (see the proof of Proposition \ref{P:TtildaOeo}), we obtain that $J_{m}$ is given by
\begin{equation*}
\begin{tikzcd}[row sep=tiny]
J_{m}: \pi_{m}(\B\PO(m))\times \pi_{m}(\B\SO(n)) \arrow[r]& \pi_{m}(\B\PO(mn))\times\pi_{m}(\B\SO(N))\\
(x,y) \arrow[r,mapsto] & \Bigl(n\esta_{1}(x)+my,2um\esta_{2}(x)+vy\Bigr)
\end{tikzcd}
\end{equation*}
where
\begin{equation*}
\begin{tikzcd}
\esta_{1}:\pi_{m-1}(\Or(m)) \arrow[r,twoheadrightarrow] & \pi_{m-1}(\Or(mn))
\end{tikzcd}
\end{equation*}
and
\begin{equation*}
\begin{tikzcd}
\esta_{2}:\pi_{m-1}(\Or(m))\arrow[r,twoheadrightarrow] &\pi_{m-1}(\Or(m^{2}))
\end{tikzcd}
\end{equation*}
are epimorphisms. 

\begin{enumerate}
\item Observe that for $m=1$ or $m\equiv 3, 5, 6, 7$ (mod 8), the homomorphism $J_{m}$ has trivial target.

\item Let $m\equiv 0,4$ (mod 8). Then $J_{m}: (\ZZ\oplus \ZZ)\times \ZZ \to \ZZ\times\ZZ$  is given by
\[
J_{m}(x\oplus y,z)=\bigl(n\esta_{1}(x\oplus y)+mz,2um\esta_{2}(x\oplus y)+vz\bigr).
\]
Note that $J_{m}$ factors as 
\begin{equation*}
\begin{tikzcd}[column sep=huge]
(\ZZ\oplus\ZZ)\times \ZZ \arrow[r,"(\esta\text{,}\id)"] & \ZZ\times\ZZ \arrow[r,"\begin{pmatrix} n \;\;\;\; 2um \\ m \quad\;\;\; v \end{pmatrix}"] & \ZZ\times\ZZ.
\end{tikzcd}
\end{equation*}
Hence $J_{m}$ is an epimorphism.

\item Let $m\equiv 2$ (mod 8). Then $J_{m}: (\ZZ/2\oplus \ZZ)\times \ZZ/2 \to \ZZ/2\times\ZZ/2$  is given by
\[
J_{m}(x\oplus y,z)=\bigl(\esta_{1}(x\oplus y),z\bigr).
\]
Then $J_{m}$ is an epimorphism.
\end{enumerate}

If we suppose $n<m$, then in the same manner it can be proved that
\begin{equation*}
\begin{tikzcd}[row sep=tiny]
J_{n}: \pi_{n}(\B\PO(m))\times \pi_{n}(\B\SO(n)) \arrow[r]& \pi_{n}(\B\SO(mn))\times\pi_{n}(\B\Or(N))\\
(x,y) \arrow[r,mapsto] & \bigl(nx+m\esta_{1}(y),2umx+v\esta_{2}(y)\bigr)
\end{tikzcd}
\end{equation*}
is an epimorphism, where
\begin{equation*}
\begin{tikzcd}
\esta_{1}:\pi_{n-1}(\SO(n)) \arrow[r,twoheadrightarrow] & \pi_{n-1}(\SO(mn))
\end{tikzcd}
\end{equation*}
and
\begin{equation*}
\begin{tikzcd}
\esta_{2}:\pi_{n-1}(\SO(n))\arrow[r,twoheadrightarrow] &\pi_{n-1}(\SO(vn)).
\end{tikzcd}
\end{equation*}
\end{proof}

From Propositions \ref{P:Jeoim}, and \ref{P:Jeom} we obtain Corollary \ref{C:Jeo}.

\begin{corollary}\label{C:Jeo}
Let $m$ and $n$ be relatively prime positive integers such that $m$ is even, and $n$ is odd. Let $d$ denote $\min\{m,n\}$. The map $J$ is $d$-connected.
\end{corollary}

\subsection{Factorization through \texorpdfstring{$f_{\tensor}:\B\fPO{m}\times\B\fSO{n}\;\;\rightarrow\;\;\B\fPO{mn}$}{f:BPOm x BSOn-->BPOmn}}

\begin{proof}[Proof of Theorem \ref{mainO}]
Diagrammatically speaking, we want to find a map 
\begin{equation*}
\begin{tikzcd}
\cA_{m}\times\cA_{n}:X \arrow[r] &  \B\PO(m)\times\B\SO(n)
\end{tikzcd}
\end{equation*}
such that diagram \eqref{cd:lpO} commutes up to homotopy
\begin{equation}\label{cd:lpO}
\begin{tikzcd}[execute at begin picture={\useasboundingbox (-4.5,-1) rectangle (4.5,1);},row sep=large,column sep=huge]
& \B\PO(m)\times\B\SO(n) \arrow[d,"f_{\tensor}"] \\
 X \arrow[r,"\cA"] \arrow[ur,dotted,"\cA_{m}\times\cA_{n}",bend left=20]  & \B\PO(mn).
\end{tikzcd}
\end{equation}

We establish the result for $m<n$; the case of $n<m$ follows similarly. Corollary \ref{C:Jeo} yields a map $J: \B\PO(m)\times\B\SO(n) \rightarrow \B\PO(mn)\times\B\SO(N)$ where $N$ is some positive integer so that $N\gg n>m$. Observe that $f_{\tensor}$ factors through $\B\PO(mn)\times\B\SO(N)$, so we can write $f_{\tensor}$ as the composite of $J$ and the projection $\proj_{1}$ shown in diagram \eqref{cd:JprojO}.

\begin{equation}\label{cd:JprojO}
\begin{tikzcd}[row sep=large,column sep=huge]
\B\PO(m)\times\B\SO(n) \arrow[r,"J"] \arrow[d,"f_{\tensor}"] & \B\PO(mn)\times\B\SO(N) \\
\B\PO(mn) \arrow[ur,leftarrow,"\proj_{1}"]& 
\end{tikzcd}
\end{equation}

Since $J$ is $m$-connected and $\dim(X)\leq m$, then by Whitehead's theorem
\begin{equation*}
\begin{tikzcd}
J_{\#}:\bigl[X, \B\PO(m)\times\B\SO(n)\bigr] \arrow[r] & \bigl[X,\B\PO(mn)\times\B\SO(N)\bigr]
\end{tikzcd}
\end{equation*}
is a surjection, \cite[Corollary 7.6.23]{SpaAT2012}.

Let $s$ denote a section of $\proj_{1}$. The surjectivity of $J_{\#}$ implies $s\circ \cA$ has a preimage $\cA_{m}\times\cA_{n}:X \rightarrow \B\PO(m)\times\B\SO(n)$ such that $J\circ (\cA_{m}\times\cA_{n})\simeq s\circ \cA$.

Commutativity of diagram \eqref{cd:lpO} follows from commutativity of diagram \eqref{cd:JprojO}. Thus, the result follows.
\end{proof}

The map $f_{\tensor}:\B\PO(m)\times\B\SO(n)\to\B\PO(mn)$ does not necessarily have a section. 
\begin{proposition}\label{nosection}
Let $m$ and $n$ be positive integers such that $m \in \{2, 4, 6, 8, 10\}$, $n$ is odd and $n>16$. Then $f_{\tensor}:\B\PO(m)\times\B\SO(n)\to\B\PO(mn)$ does not have a section.
\end{proposition}
\begin{proof}
Suppose $f_{\tensor}$ has a section, namely $\sigma$. By Proposition \ref{P:tppiiOeo}, the homomorphism induced on homotopy groups by $f_{\tensor}$ is given by $(x,y)\mapsto n\esta(x)+m\esta(y)$ for all $x\in \pi_{i}(\B\PO(m))$, $y\in\pi_{i}(\B\SO(n))$, and $i>2$.

Let $C$ denote the set $\bigl\{(4,2), (8,4), (12,6), (16,8), (16,10)\bigr\}$. From \cite{MiTo63}, \cite[Table 6.VII, Appendix A]{edm93} and Table \ref{tab:lowhg}, $\pi_{i}(\B\PO(m))$ is torsion and $\pi_{i}(\B\PO(n))\iso\pi_{i}(\B\PO(mn))\iso \ZZ$ for all $(i,m) \in C$. Since $\sigma$ is a section of $f_{\tensor}$, the image of $(f_{\tensor})_{*}\circ \sigma_{*}$ is equal to $\pi_{i}(\B\PO(mn))\iso\ZZ$ for $(i,m) \in C$. However, the image of $(f_{\tensor})_{i}:\pi_{i}(\B\PO(m))\times\pi_{i}(\B\SO(n))\to\pi_{i}(\B\PO(mn))$ is $m\ZZ$ for all $(i,m) \in C$. 
\end{proof}

If $\cA$ is a topological Azumaya algebra of degree $mn$ with an orthogonal involution over a finite CW complex of dimension higher than $\min\{m,n\}$, then $\cA$ may not have a decomposition as the one in Theorem \ref{mainO}.

\begin{example}
Let $m$ and $n$ be positive integers such that $m \in \{2, 4, 6, 8, 10\}$, $n$ is odd and $n>16$. Let $\cS$ be a topological Azumaya algebra of degree $mn$ with an orthogonal involution on the $i$-sphere $S^{i}$ such that $\cS$ generates $\pi_{i}(\B\PO_{mn})$ for $(i,m) \in C$. An argument similar to the one used in Proposition \ref{nosection} can be used to prove that $\cS$ cannot be decomposed as $\cA_{m}\tensor \cA_{n}$ for $\cA_{m}$ and $\cA_{n}$ topological Azumaya algebras of degrees $m$ and $n$, respectively, with orthogonal involutions.
\end{example}

\section{Tensor Product Decomposition of Orthogonal Bundles} \label{sec:mainObundles}

\begin{proposition}\label{P:TtildaOo}
Let $m$ and $n$ be relatively prime positive integers such that $m$ and $n$ are odd. Let $d$ denote $\min\{m,n\}$. There exists a homomorphism $\Tr: \SO(m)\times \SO(n)\rightarrow \SO(N)$, for some positive integer $N$, such that for all $i<d-1$ the homomorphisms induced on homotopy groups 
\begin{equation*}
\begin{tikzcd}
\Tr_{i}:\pi_{i}(\SO(m))\times \pi_{i}(\SO(n)) \arrow[r] & \pi_{i}(\SO(N))
\end{tikzcd}
\end{equation*}
are given by $\Tr_{i}(x,y)=ux+vy$
where $x \in \pi_{i}(\SO(m))$, $y \in \pi_{i}(\SO(n))$, and $u, v$ are some positive integers, independent of $i$, for which $|vn-um|=1$.
\end{proposition}
\begin{proof}
Assume without loss of generality $m<n$. Since $m$ and $n$ are relatively prime, there exist positive integers $u$ and $v$ such that $vn-um=\pm1$. Let $N$ denote $um+vn$, and let $\Tr$ denote the composite
\begin{equation*}
\begin{tikzcd}[column sep=large]
\SO(m)\times \SO(n) \arrow[r,"(\oplus^{u}\text{,}\oplus^{v})"] &
\SO(um)\times \SO(vn) \arrow[r,"\oplus"] &
\SO(N).
\end{tikzcd}
\end{equation*}
It follows from Corollaries \ref{C:dsO} and \ref{C:rdsO} we have that $\Tr_{i}(x,y)=ux+vy$ for all $i<m-1$.
\end{proof}

When $m$ and $n$ are odd integers, the map  $J$ defined in Section \ref{sec:mainO} becomes
\begin{equation*}
\begin{tikzcd}[row sep=tiny]
J:\B\SO(m)\times\B\SO(n) \arrow[r,mapsto] & \B\SO(mn)\times\B\SO(N)\\
(x,y) \arrow[r,mapsto] & \Bigl(f_{\tensor}(x,y)\text{,}\B\widetilde{\Tr}(x,y)\Bigr).
\end{tikzcd}
\end{equation*}

\begin{proposition}\label{P:Jo}
Let $m$ and $n$ be relatively prime positive integers such that $m$ and $n$ are odd. The  map $J$ is $d$-connected where $d$ denotes $\min\{m,n\}$.
\end{proposition}
\begin{proof}
Without loss of generality, let $m<n$. Let $i<m$. By Propositions \ref{P:tppiiOo} and \ref{P:TtildaOo}, and Corollary \ref{C:tpO} the homomorphism
\begin{equation*}
J_{i}: \pi_{i}(\B\SO(m))\times \pi_{i}(\B\SO(n)) \rightarrow \pi_{i}(\B\SO(mn))\times\pi_{i}(\B\SO(N))
\end{equation*}
is represented by the matrix $\displaystyle{\begin{pmatrix}
n & m\\
u & v
\end{pmatrix}}$.

Observe that this matrix is invertible for all $i<m$ because by Proposition \ref{P:TtildaOo} its determinant satisfies $nv-um=\pm 1$.

Let $i=m$. Propositions \ref{P:rdsO} and \ref{P:TtildaOo}, and Corollary \ref{C:rdsO} show that the homomorphism $\Tr_{m-1}:\pi_{m-1}(\SO(m))\times\pi_{m-1}(\SO(n)) \to \pi_{m-1}(\SO(N))$ is given by $\Tr_{m-1}(x,y)=\esta(u\esta(x))+\esta(v\esta(y))=u\esta(x)+vy$, where $vn-um=\pm1$. This can be seen by observing that $\Tr_{m-1}$ is equal to the sum of the two paths around diagram \eqref{cd:stabTodd}.

\begin{equation}\label{cd:stabTodd}
\begin{tikzcd}[row sep=large,column sep=tiny]
&\pi_{m-1}(\SO(m))\times\pi_{m-1}(\SO(n))  \arrow[dr,twoheadrightarrow,"\proj_{2}"] & \\
\pi_{m-1}(\SO(m)) \arrow[ru,twoheadleftarrow,"\proj_{1}"] \arrow[d,twoheadrightarrow,"\esta"] & &\pi_{m-1}(\SO(n)) \arrow[d,rightarrow,"\esta"]\\
\pi_{m-1}(\SO(um)) & & \pi_{m-1}(\SO(vn)) \arrow[dl,rightarrow,"\times v"] \arrow[u,leftarrow,"\iso"]\\
& \pi_{m-1}(\SO(um))\times \pi_{m-1}(\SO(vn)) \arrow[ul,leftarrow,"\times u"] \arrow[d,rightarrow,"\esta\times\esta"]  &\\
&\pi_{m-1}(\SO(N))\times\pi_{m-1}(\SO(N)) \arrow[d,rightarrow,"+"] \arrow[u,leftarrow,"\iso"]& \\
&\pi_{m-1}(\SO(N))&  
\end{tikzcd}
\end{equation}

Hence $J_{m}$ is given by
\begin{equation*}
\begin{tikzcd}[row sep=tiny]
J_{m}: \pi_{m}(\B\SO(m))\times \pi_{m}(\B\SO(n)) \arrow[r]& \pi_{m}(\B\SO(mn))\times\pi_{m}(\B\Or_{N})\\
(x,y) \arrow[r,mapsto] & \bigl(n\esta_{1}(x)+my,u\esta_{2}(x)+vy\bigr)
\end{tikzcd}
\end{equation*}
where 
\begin{equation*}
\begin{tikzcd}
\esta_{1}:\pi_{m-1}(\SO(m)) \arrow[r,twoheadrightarrow] & \pi_{m-1}(\SO(mn))
\end{tikzcd}
\end{equation*}
and
\begin{equation*}
\begin{tikzcd}
\esta_{2}:\pi_{m-1}(\SO(m))\arrow[r,twoheadrightarrow] &\pi_{m-1}(\SO(um))
\end{tikzcd}
\end{equation*}
are epimorphisms.

If $m=1$ or $m\equiv 3,5,7$ (mod 8), then $J_{m}$ has trivial target.

Let $m\neq1$ and $m\equiv 1$ (mod 8), then 
\[
J_{m}: (\ZZ/2\oplus\ZZ/2)\times \ZZ/2 \to \ZZ/2\times\ZZ/2
\]
is given by
\begin{align*}
J_{m}(x\oplus y,z)&=\bigl(\esta_{1}(x\oplus y)+z,u\esta_{2}(x\oplus y)+vz\bigr)
\end{align*}
because $m$ and $n$ are odd. Observe that $J_{m}$ factors as
\begin{equation*}
\begin{tikzcd}[column sep=huge]
(\ZZ/2\oplus\ZZ/2)\times \ZZ/2 \arrow[r,"(\esta\text{,}\id)"] & \ZZ/2\times\ZZ/2 \arrow[r,"\begin{pmatrix} n \quad m \\ u \quad v \end{pmatrix}"] & \ZZ/2\times\ZZ/2.
\end{tikzcd}
\end{equation*}
Hence $J_{m}$ is an epimorphism.
\end{proof}

\begin{theorem}\label{T:OrthoBundles}
Let $m$ and $n$ be relatively prime positive integers such that $m$ and $n$ are odd. Let $X$ be a CW complex such that $\dim(X)\leq \min\{m,n\}$. If $\cV$ is an orthogonal complex vector bundle of rank $mn$ over $X$, then there exist orthogonal complex vector bundles $\cV_{m}$ and $\cV_{n}$ of ranks $m$ and $n$, respectively, such that $\cV\iso \cV_{m}\tensor\cV_{n}$. 
\end{theorem}
\begin{proof}
We can use Propositions \ref{P:TtildaOo} and \ref{P:Jo} to prove the result in a similar manner to the proof of Theorem \ref{mainO}.
\end{proof}

\newpage
\bibliographystyle{alpha}
\bibliography{MyTAABiblio}

@book{edm93,
author = {Mathematical Society of Japan, Corporate and It\^{o}, Kiyosi},
title = {Encyclopedic Dictionary of Mathematics (2nd Ed.)},
year = {1993},
isbn = {0262590204},
publisher = {MIT Press},
address = {Cambridge, MA, USA}
}

@article{AW2x32014,
author = {Antieau, Benjamin and Williams, Ben},
doi = {10.1007/s00222-013-0479-7},
ISSN = {1432-1297},
journal = {Invent. Math.},
number = {1},
pages = {47--56},
title = {{Unramified division algebras do not always contain Azumaya maximal orders}},
volume = {197},
year={2014}
}

@article{TAA2021, 
title={{Decomposition of topological Azumaya algebras}}, 
volume={65}, 
DOI={10.4153/S000843952100045X}, 
number={2}, 
journal={Canadian Mathematical Bulletin}, 
publisher={Canadian Mathematical Society}, 
author={Arcila-Maya, Niny}, 
year={2022}, 
pages={506–524}
}

@article{AG1960,
author = {Auslander, Maurice and Goldman, Oscar},
DOI = {10.1090/S0002-9947-1960-0121392-6},
ISSN = {00029947},
journal = {Trans. Amer. Math. Soc},
number = {3},
pages = {367--409},
publisher = {American Mathematical Society},
title = {{The Brauer group of a commutative ring}},
volume = {97},
year = {1960}
}

@article{ART1979,
author = {S. A. Amitsur and L. H. Rowen and J. P. Tignol},
title = {{Division algebras of degree 4 and 8 with involution}},
volume = {1},
journal = {Bulletin (New Series) of the American Mathematical Society},
number = {4},
publisher = {American Mathematical Society},
pages = {691 -- 693},
year = {1979},
doi = {bams/1183544588},
URL = {https://doi.org/}
}

@article{Azu1951,
author = {Azumaya, Gor{\^{o}}},
DOI = {10.1017/S0027763000010114},
journal = {Nagoya Math. J.},
pages = {119--150},
publisher = {Cambridge University Press},
title = {{On maximally central algebras}},
volume = {2},
year = {1951}
}

@book{Car2017,
  title={Groups, Matrices, and Vector Spaces: A Group Theoretic Approach to Linear Algebra},
  author={Carrell, J.B.},
  isbn={9780387794280},
  year={2017},
  publisher={Springer New York}
}

@book {FD1993,
    AUTHOR = {Farb, Benson and Dennis, R. Keith},
     TITLE = {Noncommutative algebra},
    SERIES = {Graduate Texts in Mathematics},
    VOLUME = {144},
 PUBLISHER = {Springer-Verlag, New York},
      YEAR = {1993},
     PAGES = {xiv+223},
      ISBN = {0-387-94057-X},
   MRCLASS = {16-01},
  MRNUMBER = {1233388},
       DOI = {10.1007/978-1-4612-0889-1},
}

@book{GS06, 
place={Cambridge}, 
series={Cambridge Studies in Advanced Mathematics}, 
title={{Central simple algebras and Galois cohomology}}, 
DOI={10.1017/CBO9780511607219}, 
publisher={Cambridge University Press}, 
author={Gille, Philippe and Szamuely, Tamás}, 
year={2006}, 
collection={Cambridge Studies in Advanced Mathematics}
}

@book{Gre1967,
    AUTHOR = {Greub, W. H.},
     TITLE = {Multilinear algebra},
    SERIES = {Die Grundlehren der Mathematischen Wissenschaften, Band 136},
 PUBLISHER = {Springer-Verlag New York, Inc., New York},
      YEAR = {1967},
     PAGES = {x+225},
   MRCLASS = {15.00},
  MRNUMBER = {0224623},
MRREVIEWER = {H. Gross},
}

@incollection{GroI1966,
author = {Grothendieck, Alexander},
title = {{Le groupe de Brauer: I. Alg{\`{e}}bres d'Azumaya et interpr{\'{e}}tations diverses}},
booktitle = {S{\'{e}}minaire Bourbaki: ann{\'{e}}es 1964/65 1965/66, expos{\'{e}}s 277-312},
series = {S{\'{e}}minaire Bourbaki},
publisher = {Soci{\'{e}}t{\'{e}} math{\'{e}}matique de France},
number = {9},
year = {1966},
note = {talk:290},
pages = {199--219},
mrnumber = {1608798},
}

@book{KMRT1998,
    AUTHOR = {Knus, Max-Albert and Merkur'ev, Alexander and Rost, Markus and
              Tignol, Jean-Pierre},
     TITLE = {The book of involutions},
    SERIES = {American Mathematical Society Colloquium Publications},
    VOLUME = {44},
      NOTE = {With a preface in French by J. Tits},
 PUBLISHER = {American Mathematical Society, Providence, RI},
      YEAR = {1998},
     PAGES = {xxii+593},
      ISBN = {0-8218-0904-0},
   MRCLASS = {16K20 (11E39 11E57 11E72 11E88 16W10 20G10)},
  MRNUMBER = {1632779},
MRREVIEWER = {A. R. Wadsworth},
       DOI = {10.1090/coll/044},
}

@article{KPS1990,
title = {Azumaya algebras with involutions},
journal = {Journal of Algebra},
volume = {130},
number = {1},
pages = {65-82},
year = {1990},
issn = {0021-8693},
doi = {10.1016/0021-8693(90)90100-3},
author = {Max-Albert Knus and R. Parimala and V. Srinivas}
}

@article{KPS1991,
title = {{Involutions on rank 16 central simple algebras}},
journal = {Journal of the Indian Math. Soc.},
volume = {57},
number = {1},
pages = {143-151},
year = {1991},
issn = {},
doi = {},
author = {Max-Albert Knus and R. Parimala and V. Srinivas}
}

@Article{Merku,
 Author = {Merkurjev, A. S.},
 Title = {{On the norm residue symbol of degree 2}},
 FJournal = {Soviet Mathematics. Doklady},
 Journal = {Dokl. Akad. Nauk SSSR},
 Volume = {261},
 Pages = {542--547},
 Year = {1981},
 Language = {English}
}

@book{MiToToG1991,
    AUTHOR = {Mimura, Mamoru and Toda, Hirosi},
     TITLE = {Topology of {L}ie groups. {I}, {II}},
    SERIES = {Translations of Mathematical Monographs},
    VOLUME = {91},
      NOTE = {Translated from the 1978 Japanese edition by the authors},
 PUBLISHER = {American Mathematical Society, Providence, RI},
      YEAR = {1991},
     PAGES = {iv+451},
      ISBN = {0-8218-4541-1},
   MRCLASS = {55-02 (22-02 55P99 57T20 57T25)},
  MRNUMBER = {1122592},
MRREVIEWER = {V. P. Snaith},
       DOI = {10.1090/mmono/091},
}

@article{MiTo63,
author = {Mimura, Mamoru and Toda, Hirosi},
title = {{Homotopy Groups of $SU(3)$, $SU(4)$ and $Sp(2)$}},
volume = {3},
journal = {Journal of Mathematics of Kyoto University},
number = {2},
publisher = {Duke University Press},
pages = {217 -- 250},
year = {1963},
doi = {10.1215/kjm/1250524818},
URL = {https://doi.org/10.1215/kjm/1250524818}
}

@article {SAzuInvo1978,
    AUTHOR = {Saltman, David J.},
     TITLE = {Azumaya algebras with involution},
   JOURNAL = {J. Algebra},
  FJOURNAL = {Journal of Algebra},
    VOLUME = {52},
      YEAR = {1978},
    NUMBER = {2},
     PAGES = {526--539},
      ISSN = {0021-8693},
   MRCLASS = {16A16 (16A28)},
  MRNUMBER = {495234},
MRREVIEWER = {Lindsay N. Childs},
       DOI = {10.1016/0021-8693(78)90253-3},
}

@book {Sdiv1999,
author = {Saltman, David J},
title = {{Lectures on division algebras}},
series = {CBMS Reg. Conf. Ser. Math.},
volume = {94},
publisher = {AMS, Providence, RI; on behalf of CBMS, Washington, DC},
year = {1999},
pages = {viii+120},
ISBN = {0-8218-0979-2},
DOI = {10.1090/cbms/094},
}

@book{SpaAT2012,
address = {New York, NY},
author = {Spanier, Edwin H.},
booktitle = {Springer-Verlag New York},
DOI = {10.1007/978-1-4684-9322-1},
ISBN = {978-0-387-94426-5},
edition = {Illustrate},
pages = {XIV, 548},
publisher = {Springer New York},
title = {{Algebraic Topology}},
year = {1981}
}

@book{Steen1951,
 DOI = {10.1515/9781400883875},
 ISBN = {9780691005485},
 author = {Steenrod, Norman},
 publisher = {Princeton University Press},
 title = {The Topology of Fibre Bundles. (PMS-14)},
 year = {1951}
}

@Book{WhiHT1978,
author = { Whitehead, George W. },
title = { Elements of homotopy theory},
DOI = {10.1007/978-1-4612-6318-0},
isbn = {978-1-4612-6318-0},
publisher = { Springer-Verlag New York },
pages = { xxi, 746},
volume = {61},
year = { 1978 },
}
\end{document}